\documentclass[12pt, a4paper, reqno]{amsart}

\usepackage[%
    a4paper,
	left=2.5cm,       
	right=2.5cm,      
	top=3.5cm,        
	bottom=3.5cm,     
	heightrounded,    
	bindingoffset=0mm 
]{geometry}

\usepackage[T1]{fontenc}
\usepackage[utf8]{inputenc}
\usepackage[english]{babel}
\usepackage{microtype}
\usepackage{lmodern}
\usepackage{amsmath,amsthm,amssymb,mathrsfs}
\usepackage{enumerate} 
\usepackage{graphicx}

\usepackage{color}
\usepackage{graphicx}
\usepackage{tikz}

\newcommand{\bq}{\begin{equation}}
\newcommand{\eq}{\end{equation}}
\newcommand{\bqa}{\begin{eqnarray*}}
\newcommand{\eqa}{\end{eqnarray*}}

\usepackage{hyperref}
\hypersetup{colorlinks={true},linkcolor={blue},citecolor=blue}
\usepackage[nameinlink,noabbrev,capitalise]{cleveref}
\theoremstyle{plain}
\newtheorem{theorem}{Theorem}[section]
\newtheorem{Ntheorem}{Theorem}

\newtheorem{proposition}[theorem]{Proposition}
\newtheorem{lemma}[theorem]{Lemma}
\newtheorem{corollary}[theorem]{Corollary}

\theoremstyle{definition}

\crefname{lemma}{Lemma}{Lemmas}
\crefname{proposition}{Proposition}{Propositions}
\theoremstyle{remark}
\newtheorem{remark}[theorem]{Remark} 

\DeclareSymbolFont{pletters}{OT1}{cmr}{m}{sl}
\DeclareMathSymbol{s}{\mathalpha}{pletters}{`s}

\def\L1{\mathcal{L}^{(1)}}
\def\L2{\mathcal{L}^{(2)}}
\def\L3{\mathcal{L}^{(3)}}

\numberwithin{equation}{section}

\renewcommand{\geq}{\geqslant}
\renewcommand{\leq}{\leqslant} 

\makeatletter
\renewcommand{\@@and}{\&}
\makeatother

\title[Interpolation on M{\"u}ntz Spaces]{Marcinkiewicz Type Theorems for Interpolation of Operators Acting on M{\"u}ntz Spaces}
\date{}

\author{Micka\"el Latocca}
\address{
Laboratoire de Mathématiques et de Modélisation d'\'Evry (LaMME)\\
Université d'\'Evry\\
23 Bd François Mitterrand, 91000 Évry-Courcouronnes, France
}
\email[M. Latocca]{mickael.latocca@univ-evry.fr}

\author{Vincent Munnier}
\address{Lycée Jacques Prévert 163 rue de Billancourt – 92100 Boulogne-Billancourt, France
}
\email[V. Munnier]{munniervincent@hotmail.fr}

\begin{document}

\begin{abstract} We prove interpolation results in the spirit of the Marcinkiewicz theorem. The operators considered in this article are defined on M{\"u}ntz spaces, which are not dense subspaces of $L^p$, and for which the classical interpolation theory cannot be applied directly. Our proofs crucially rely on strong decoupling of $L^p$ norms, a that was first observed by Gurariy--Macaev and later generalized.
\end{abstract}

\maketitle

\section{Introduction}

In this article we fix an increasing sequence of positive real numbers $\Lambda = \{\lambda_{k}\}_{k\geq 0}$ such that $\sum_{k\geq 0} \frac{1}{\lambda_k} < \infty$. In view of the Müntz--Sasz theorem \cite[p. 172]{BE}, the space $M_{\Lambda}$ defined by 
\[
    M_{\Lambda} = \left\{ f : [0,1] \to \mathbb{R}: f(t)=\sum_{k=0}^Ka_kt^{\lambda_k}, a_k \in \mathbb{C}, K\geqslant 0\right\},
\]
is \textit{not} dense in $C^0([0,1])$. We refer to \cite{GL} for a precise account of important properties of these spaces. We impose further that the sequence $\Lambda$ satisfies the following condition which we refer to as $\Lambda$ being \textit{quasi-lacunary}: there exists $N\geq 1$ and $q>1$ such that $\Lambda$ is the disjoint union of the $E_k=\{\lambda_{n_k+1}, \dots, \lambda_{n_{k+1}}\}$, that $\# E_k \leq N$ and $\frac{\lambda_{n_{k+1}}}{\lambda_{n_k}} > q$ for $k$ large enough. We recover the case of a lacunary sequence when $N=1$. Morover we say that $\Lambda$ is \textit{quasi-lacunary and subgeometric} if in addition for some $q'>1$ there holds $\frac{\lambda_{n_{k+1}}}{\lambda_{n_k}} \leq q'$ for all $k$ large enough. 
In the following, we will use the notation $F_{k}=\operatorname{Span}\{t^{\lambda}, \lambda \in E_k \}$. 

If $\mu$ is a positive Borel measure, we write $L_{\Lambda}^r(\mathrm{d}\mu)$ for the closure of $M_{\Lambda}$ under the usual $L^r(\mathrm{d}\mu)$ norm. When $\mu$ is the Lebesgue measure we simply write $L_{\Lambda}^r$. We insist on the fact that $L^r_{\Lambda}(\mathrm{d}\mu) \varsubsetneq L^r(\mathrm{d}\mu)$. We also denote by $L^{p}_{\Lambda,\alpha}$ the closure of $M_{\Lambda}$ under the $L^p(\mathrm{d}\nu_{\alpha - 1})$ norm, where 
$\mathrm{d}\nu_{\alpha - 1} := (1-x)^{\alpha - 1}\mathrm{d}x$ on $[0,1]$. In particular $M_{\Lambda,1}^p=M_{\Lambda}^p$.

The main focus of this article is to investigate how one can adapt real-interpolation methods in the context of Müntz spaces for which the usual interpolation theory does not apply as we are working on non dense subspaces. Our work is in line with those studying the action of operators in M{\"u}ntz spaces \cite{AGHL,AL,BW09,L18,A09}. We will illustrate how these interpolation results can be used to derive generalizations of embedding results previously obtained in \cite{LM24,GaLe,NT,CFT}. These are embeddings results of the form $L_{\Lambda,\alpha}^{p/\beta} \hookrightarrow L_{\Lambda}^p(\mathrm{d}\mu)$ where $\mu$ satisfies some geometric conditions. 

\subsection{Main results: interpolation of Müntz spaces}

Let $T$ be a sublinear operator acting on $M_{\Lambda}$, that is satisfying $T(cf)=|c||Tf|$ for $c\in\mathbb{C}$, $|T(f+g)| \leq |Tf| + |Tg|$ and $|Tf - Tg| \leq |T(f-g)|$ for all $f, g \in M_{\Lambda}$.
In the case $T=\operatorname{id}$ a condition for the boundedness of $\operatorname{id} : M_{\Lambda}^{\frac{r}{\beta}} \longrightarrow L^r(\mathrm{d}\mu)$ has been obtained in \cite{CFT,NT,GaLe,LM24}; we now want to investigate how such continuity of more general operators $T$ can be derived from weak continuity of the operator $T : M_{\Lambda}^{\frac{p}{\beta}} \longrightarrow L^p(\mathrm{d}\mu)$ and $T : M_{\Lambda}^{\frac{q}{\beta}} \longrightarrow L^q(\mathrm{d}\mu)$ for some $p<r<q$, which is the classical problem of interpolation. 

Given $\alpha, \beta >0$ and a positive Borel measure $\mu$ on $[0,1]$, we say that an operator $T : M_{\Lambda} \to$ is: 
\begin{itemize}
    \item of strong type $r$ if there is a constant $C_r>0$ such that for all $f\in M_{\Lambda}$ it holds 
    \begin{equation}
        \label{eq.strong-cont}
        \|T(f)\|_{L^r(\mathrm{d}\mu)} \leq C_r \|f\|_{L^{\frac{r}{\beta}}(\mathrm{d}\nu_{\alpha - 1})}
    \end{equation}
    i.e. $T : M_{\Lambda}^{\frac{r}{\beta}}(\mathrm{d}\nu_{\alpha-1}) \longrightarrow L^r(\mathrm{d\mu})$ is continuous;
    \item of weak type $r$ if there is a constant $C_r>0$ such that for all $f\in M_{\Lambda}$ it holds 
    \begin{equation}
        \label{eq.weak-cont}
        \mu\left(|Tf|>\lambda\right) \leq C_r^r \lambda^{-r}\|f\|^r_{L^{\frac{r}{\beta}}(\mathrm{d}\nu_{\alpha - 1})};
    \end{equation}
    \item of restricted strong type $r$ if for all $k$, there is a constant $C_r(k)>0$ such that for all $f_k\in F_k$ the estimate \eqref{eq.strong-cont} holds;
    \item of restricted weak type $r$ if for all $k$, there is a constant $C_r(k)>0$ such that for all $f_k\in F_k$ the estimate \eqref{eq.weak-cont} holds.
\end{itemize} 
We recall that strong type $r$ bounds imply weak type $r$ bounds thanks to Markov's inequality.

When $\beta \geqslant 1$, we prove the following interpolation theorem, which statement is akin to that of Marcinkiewicz. 

\begin{Ntheorem}\label{thm.main-marcinkiewicz}
Let $\alpha>0$, $\beta\geq 1$, $\mu$ be a positive Borel measure on $[0,1]$ and $\Lambda$ be a quasi-lacunary sequence. 
Let $(p,q)$ such that $\beta\leq p < q <+\infty$, $r \in (p,q)$ and $\theta\in (0,1)$ such that $\frac{1}{r} = \frac{1-\theta}{p} + \frac{\theta}{q}$. 
Let $T$ be a sublinear operator acting on $M_{\Lambda}$. Assume that $T$ is of restricted weak type $p$ and $q$ with $C_p(k)\leqslant C_p< \infty$ and $C_q(k)\leqslant C_q <\infty$. Then: 
\begin{enumerate}[(i)]
    \item $T$ is of strong type $r$, and $C_r \leqslant C_p^{1-\theta}C_q^{\theta}$.
    \item If moreover $C_p(k) \to 0$ or $C_q(k)\to 0$ as $k\to \infty$ then $T$ is a compact operator.
\end{enumerate}
\end{Ntheorem}

\begin{remark} It is important to remark that such an interpolation result does not easily follow from the classical Marcinkiewicz interpolation theorem, as the space $M_{\Lambda}^{\frac{p}{\beta}}$ is \text{not} dense in $L^{\frac{p}{\beta}}_{\Lambda}$ and that interpolation of subspaces is in general a difficult matter.  
\end{remark}

\begin{remark} In \cref{sec.C-example-subcritical} we construct an operator $T$ that is of restricted strong type $r>1$ (therefore of restricted weak type $r$) with constants $C_r(k) \leq C_r$, but not of strong type $r$, meaning that one cannot take $p=q=r>1$ in the above statement. The theorem is therefore optimal in general. We also provide an example of application of the above theorem to illustrate why it is non-trivial, see \cref{sec.example-subcritical}.
\end{remark}

\begin{remark}[Limiting case]\label{rem:strong}
If $T$ is of restricted strong type $\beta$ then it is easily seen that $T$ is then of strong type $\beta$. Observe indeed that for all $f = \sum_k f_k \in M_{\Lambda}$ and because $\beta \geq 1$, the triangle inequality and the Minkowski inequality imply
\[
    \|Tf\|_{L^{\beta}(\mathrm{d}\mu)} \lesssim \|\{Tf_k\}_{k\geq 0}\|_{L^{\beta}(\mathrm{d}\mu)\ell^1_k} \lesssim \|\{Tf_k\}_{k\geq 0}\|_{\ell^1_kL^{\beta}(\mathrm{d}\mu)}.
\]
Then, because of the restricted strong-type continuity assumption on $T$, it follows that 
\[
    \|Tf\|_{L^{\beta}(\mathrm{d}\mu)} \lesssim \sum_{k\geq 0}\|f_{k}\|_{L^{1}(\mathrm{d}\nu_{\alpha-1})}
        \lesssim \|f\|_{L^{1}(\mathrm{d}\nu_{\alpha -1})}
\]
where the last estimate stems for an application of \cref{thm:Lp-frames}.
\end{remark}

In the case $\beta \in (0,1)$ we prove the following interpolation result.

\begin{Ntheorem}\label{thm.interpol-supercritical}
Let $\alpha >0$, $\beta \in (0,1)$, $\mu$ be a positive Borel measure on $[0,1]$ and $\Lambda$ be a quasi-lacunary sequence. Let $(p,q)$ be such that $1<p<q<\infty$, $r\in(p,q)$ and $\theta \in (0,1)$ such that $\frac{1}{r}=\frac{1-\theta}{p} + \frac{\theta}{q}$. Let $T$ be a sublinear operator acting on $M_{\Lambda}$. Assume that $T$ is of restricted weak type $p$ and $q$, with bounds
\[
    C_p(k) \leq C_p\varepsilon_k^{\frac{1-\beta}{p}} \text{ and } C_q(k)\leqslant C_q\varepsilon_k^{\frac{1-\beta}{q}} \text{ for all } k\geq 0.
\]
Then $T$ is of strong type $r$, with $C_r \leqslant C_{\varepsilon}^{\frac{1-\beta}{r}}C_p^{1-\theta}C_q^{\theta}$, where $C_{\varepsilon}:=\displaystyle\sum_{k\geq 0} \varepsilon_k < \infty$. Moreover, $T$ is a compact operator. 
\end{Ntheorem}

\begin{remark} In \cref{sec.examples-supercritical} we provide a non-trivial example of application of our theorem, whereas in \cref{sec.counter-examples-supercritical} we construct an operator $T$ that of restricted strong type $r$ with $C_r(k) \leq C_r\varepsilon_k^{\frac{1-\beta}{p}}$ and where $\sum_{k\geq 0} \varepsilon_k<0$, but that is not of strong type $r$, therefore proving that one cannot take $p=q$ in the above. In \cref{sec.necessity} we prove that in the case of positive $T$, the condition $\sum_{k\geq 0} \varepsilon_k < \infty$ is necessary. 
\end{remark}

\subsection{Consequences of the interpolation theorems}

As the assumptions of \cref{thm.main-marcinkiewicz} and \cref{thm.interpol-supercritical} may be difficult to check in practice, we present some corollaries that generalize previous results \cite{LM24}.

\subsubsection{The subcritical case $\beta \geq 1$}\label{sec.consequence-subcritical}

The following condition will appear to be useful: let $\lambda_{n_k}$ be the lacunary sequence defined in the definition of $\Lambda$. We say that $\mu$ satisfies the $B_{p,\Lambda}(\alpha, \beta)$ condition if there holds  
\[
    \sup_{k\geq 0} \lambda_{n_k}^{\alpha\beta}\int_{[0,1]} t^{p\lambda_{n_k}}\mathrm{d}\mu(t) < \infty, 
\]
This generalizes the condition $B_{p}$ introduced in \cite{GaLe}, which corresponds to our condition $B_{p,\Lambda}(1,1)$. Remark that in the lacunary case, $B_{p,\Lambda}(\alpha,\beta)$ is nothing but the assumption of \cref{thm.main-marcinkiewicz}, and therefore a weaker assumption in the quasi-lacunary case.

In the case $T=\imath_{\mu}$, we observe that a one-sided boundedness condition is enough in order to apply \cref{thm.main-marcinkiewicz}. More precisely, we have the following:  

\begin{corollary}\label{coro:boundedness}
Let $\alpha>0$, $\beta\geq 1$ and  $p\geq \beta$. Let $\Lambda$ be a quasi-lacunary sequence and $\mu$ be a positive Borel measure on $[0,1]$. Assume that the inclusion $\imath_{\mu}$ is of restricted weak type $p$. Then, for any $r>p$, $\imath_{\mu}$ is of strong type $r$.
\end{corollary}

\begin{remark}
In particular, if $\imath_{\mu}$ is of strong-type $p\geq \beta$, it follows that $\imath_{\mu}$ if of strong type $r$ for any $r\geq p$.
\end{remark}

\begin{corollary}\label{coro.embeddings-subcritical}
Let $\Lambda$ be a quasi-lacunary sequence, with blocks of length bounded by $N\geq 1$. Let $\mu$ be a positive Borel regular measure on $[0,1]$ satisfying the condition $B_{p,\Lambda}(\alpha,\beta)$ for some $p >0$, $\alpha > 0$ and $\beta\geq 1$. Then, 
\begin{enumerate}[(i)]
    \item If $pN\leq \beta$ then $M_{\Lambda,\alpha}^{\frac{r}{\beta}} \hookrightarrow L^r(\mathrm{d}\mu)$ holds for all $r\geq \beta$; 
    \item If $pN >\beta$ then $M_{\Lambda,\alpha}^{\frac{r}{\beta}} \hookrightarrow L^r(\mathrm{d}\mu)$ holds for any $r>pN$.
\end{enumerate}
\end{corollary}

\begin{remark} In case $\lambda_{n_k}^{\alpha\beta}\int_{[0,1]} t^{p\lambda_{n_k}}\mathrm{d}\mu(t) \to 0$ as $k\to \infty$, the embeddings are compact, see \cref{rem.compact-sub}.
\end{remark}

\begin{remark}
In particular, if $\Lambda$ is a lacunary sequence, i.e. $N=1$, \cref{coro.embeddings-subcritical} extends the results of \cite{GaLe,LM24}, and when $N>1$ the above is a natural generalization.
\end{remark}

We introduce the following notation: $\mathcal{M}_{x^{\gamma}}$ stands for the set of positive Borel measures $\mu$ supported on $[0,1]$ which satisfy 
\[
	\mu([1-\varepsilon, 1])\lesssim\varepsilon^{\gamma},	
\]
where the implicit constant depends only on $\mu$. For instance, $\mathrm{d}\nu_{\gamma - 1} := (1-x)^{\gamma - 1}\mathrm{d}x \in \mathcal{M}_{x^{\gamma}}$.

In \cite[Theorem A]{LM24} it is proven that $\mu\in\mathcal{M}_{x^{\beta}}$ if and only if the following embedding $M_{\Lambda}^{\frac{p}{\beta}} \hookrightarrow L^p(\mathrm{d}\mu)$ holds for all $p\geq \beta$. The following is a generalization in the case $\alpha \neq 1$.

\begin{corollary}\label{coro.subcritical1}
    Let $\Lambda$ be a quasi-lacunary and sub-geometric sequence. Let $\beta\geq 1$ and $\alpha>0$. Then, the following are equivalent:
    \begin{enumerate}[(i)]
    \item The measure $\mu$ belongs to $\mathcal{M}_{x^{\alpha\beta}}$.
    \item For any $p\geq \beta$, there holds $M_{\Lambda,\alpha}^{\frac{p}{\beta}} \hookrightarrow L^p(\mathrm{d}\mu)$.
    \end{enumerate}
\end{corollary}

\subsubsection{The supercritical case $0<\beta <1$}

In the case $\beta > 1$ we state similar corollaries. The appropriate condition will be the following one: we say that $\mu$ satisfies the $A_{p,\Lambda}(\alpha,\beta)$ condition if 
\begin{equation*}
    \sum_{k\geq 0} \lambda_{n_k}^\frac{\beta \alpha}{1-\beta}\left(\int_{[0,1]}t^{p\lambda_{n_k}}\,\mathrm{d}\mu(t) \right)^{\frac{1}{1-\beta}}<+\infty,
\end{equation*}
where $\{\lambda_{n_k}\}_{k\geq 0}$ is the lacunary sequence extracted from $\{\lambda_k\}_{k\geq 0}$. In the lacunary case this corresponds to the summability condition $\sum_{k\geq 0}\varepsilon_k < \infty$ of \cref{thm.interpol-supercritical}, and should be viewed as a weaker requirement in the quasi-lacunary case.  

\begin{corollary}\label{th5} Let $\Lambda$ be a quasi-lacunary sequence with blocks of size bounded by $N\geq 1$. Let $\mu$ be a positive Borel measure supported on $[0,1]$ satisfying $A_{p,\Lambda}(\alpha,\beta)$ for some $p>0$, $\alpha>0$ and $0<\beta<1$. Then,
\begin{enumerate}[(i)]
    \item If $pN\leq 1$ then $M_{\Lambda,\alpha}^{\frac{r}{\beta}} \hookrightarrow L^r(\mathrm{d}\mu)$ holds for any $r\geq 1$;
    \item If $pN>1$ then $M_{\Lambda,\alpha}^{\frac{r}{\beta}} \hookrightarrow L^r(\mathrm{d}\mu)$ holds for any $r>pN$.
\end{enumerate}
Moreover, the embeddings are compact.
\end{corollary}

We recall that in \cite[Theorem B]{LM24} it is proven that the embedding $M_{\Lambda}^\frac{p}{\beta} \hookrightarrow L^p(\mathrm{d}\mu)$ holds for all $p\geq 1$ if and only if 
\[
    \displaystyle\int_{[0,1]}\Big(\int_{[0,1]}\frac{\mathrm{d}\mu(t)}{(1-\rho t)}\Big)^{\frac{1}{1-\beta}}\,\mathrm{d}\rho<+\infty,
\]
and also equivalently to 
\[
    \sum_{k\geq 0} \lambda_{n_k}^{\frac{\beta}{1-\beta}}\Big(\int_{[0,1]}t^{p\lambda_{n_k}}\,\mathrm{d}\mu(t)\Big)^{\frac{1}{1-\beta}}<+\infty \quad \text{ for all } p\geqslant 1.
\]
The following result is a generalization of \cite[Theorem B]{LM24} to the case $\alpha \neq 1$.

\begin{corollary}\label{th6}
Let $\Lambda$ be a quasi-lacunary and subgeometric sequence. Let $\alpha>0$, $0<\beta<1$ and $1 \leq p < q \leq \infty$. Then, for any $p\geq 1$ the following are equivalent:
\begin{enumerate}[(i)]
    \item It holds $M_{\Lambda,\alpha}^{\frac{p}{\beta}} \hookrightarrow L^p(\mathrm{d}\mu)$;
    \item It holds $\displaystyle \sum_{k\geq 0} \lambda_{n_k}^\frac{\beta \alpha}{1-\beta}\left(\int_{[0,1]}t^{p\lambda_{n_k}}\,\mathrm{d}\mu(t) \right)^{\frac{1}{1-\beta}}<+\infty$.
\end{enumerate}
The embeddings are in fact compact. 
\end{corollary}
 
\section{Preliminaries to the proof}

In this section we recall some important results that we should use throughout the proofs.  

\subsection{Results on M{\"u}ntz spaces}

One of the most useful tool at our disposition is the following, which is a generalization of a theorem of Gurariy--Macaev \cite{GM}. We recall that $F_{k}=\operatorname{Span}\{t^{\lambda}, \lambda \in E_k \}$.

\begin{theorem}[\cite{LM24}, Theorem C]\label{thm:Lp-frames} Let $p\in [1,\infty)$, $\alpha > 0$ and $\mathrm{d}\nu_{\alpha-1} = (1-x)^{\alpha-1}\,\mathrm{d}x$. Then there exists two constants $C_1,C_2 >0$ such that for all $f_k \in F_k$ there holds 
    \[
        C_1\left(\sum_{k\geqslant 0}\|f_k\|_{L^p(\mathrm{d}\nu_{\alpha-1})}^p\right)^{\frac{1}{p}}\leq \left\|\sum_{k\geqslant 0}f_k\right\|_{L^p(\mathrm{d}\nu_{\alpha-1})}\leq C_2\left(\sum_{k\geqslant 0}\|f_k\|_{L^p(\mathrm{d}\nu_{\alpha-1})}^p\right)^{\frac{1}{p}}. 
    \]
\end{theorem}

The above result should be understood as a complete decoupling of $L^p(\mathrm{d\nu_{\alpha - 1}})$ norms of the atoms $f_k$ of a function $f = \sum_{k\geq 0}f_k \in M_{\Lambda}$.

Given a function $f_k\in F_k$, one can roughly estimate $f_k$ by a monomial $x^{\frac{\lambda_k+1}{N}}$:

\begin{theorem}[\cite{GL}, Corollary 8.1.2 ]\label{thm:muntz-bound} Let $\Lambda$ be a quasi-lacunary sequence. 
    Let $k \geqslant 0$. Then for any $f_k \in F_{k}$ and for all $x \in [0,1]$, there holds 
    \[
        \vert f_k(x) \vert \lesssim x^{\frac{\lambda_{n_{k}}+1}{N}}\|f_k\|_{L^{\infty}},
    \]
    where the implicit constant depends only on $q, N$, but not on $k$, nor $f_k$. 
\end{theorem}

When restricted to functions $f_k\in F_k$, one can exploit the relationship between different $L^p$ norms of $f_k$, using the following: 

\begin{proposition}[Generalized Berstein estimates, \cite{LM24}, Proposition 2.5]\label{prop:atom} 
    Let $\beta >0$, $\mu \in \mathcal{M}_{x^{\beta}}$ and $\alpha >-1$. There exists $k_0 > 0$ such that for all $k\geqslant k_0$, for all $f_{k} \in F_{k}$, and for any $p,q \geqslant 1$ there holds 
    \[
        \|f_{k}\|_{L^{p}(\mathrm{d}\mu)}\lesssim \lambda_{n_{k}}^{\delta}\|f_{k}\|_{L^{q}((1-x)^{\alpha}\mathrm{d}x)},	
    \] 
    with $\delta=\frac{1+\alpha}{q}-\frac{\beta}{p}$, and where the implicit constant only depends on $p$,$q$, $\alpha$, $\beta$ and also on $q, N$.
\end{proposition}

\subsection{An elementary interpolation inequality}

\begin{lemma}\label{interpolbis}
Assume that $T$ is a sublinear operator of restricted weak type $p$ and $q$ for some $1\leq p<q<\infty$, with constants $C_p(k)$ and $C_q(k)$. Let $r\in(p,q)$, and write $\frac{1}{r}=\frac{1-\theta}{p}+\frac{\theta}{q}$ for some $\theta\in(0,1)$.

Then $T$ is of restricted strong type $r$ with constant $C_r(k)\leqslant C_p^{1-\theta}(k)C_q^{\theta}(k)$. 
\end{lemma}

\begin{proof} The proof is elementary, we provide it for the sake of completeness. Fix $k\geq 0$ and $f_k\in F_k$. We simply write $C_p$ and $C_q$ instead of $C_p(k)$ and $C_q(k)$ as $k$ is fixed. Let also $A$ be a positive real number to be chosen later. Then the Cavalieri formula writes 
\begin{multline*}
    \|Tf_{k}\|_{L^{r}(\mathrm{d}\mu)}^{r} = r\int_{0}^{+\infty} \lambda^{r-1}\mu( \vert Tf_{k} \vert >\lambda)\,\mathrm{d}\lambda\\
    = r\int_{0}^{A}\lambda^{r-1}\mu( \vert Tf_{k} \vert > \lambda)\,\mathrm{d}\lambda + r\int_{A}^{+\infty}\lambda^{r-1}\mu( \vert Tf_{k} \vert > \lambda)\,\mathrm{d}\lambda. 
\end{multline*}
Using that $T$ is of restricted weak type $p$ and $q$, we infer  
\begin{multline*}
    \|Tf_{k}\|_{L^{r}(\mathrm{d}\mu)}^{r} \leq C_p^p\|f_{k}\|_{L^{\frac{p}{\beta}}(\mathrm{d}\nu_{\alpha - 1})}^{p}\int_{0}^{A}\lambda^{r-p-1}\,\mathrm{d}\lambda + C_q^q\|f_{k}\|_{L^{\frac{q}{\beta}}(\mathrm{d}\nu_{\alpha - 1})}^{q}\int_{A}^{+\infty}\lambda^{r-q-1}\,\mathrm{d}\lambda\\
    \lesssim C_p^pA^{r-p}\|f_{k}\|_{L^{\frac{p}{\beta}}(\mathrm{d}\nu_{\alpha - 1})}^{p}+C_q^qA^{r-q}\|f_{k}\|_{L^{\frac{q}{\beta}}(\mathrm{d}\nu_{\alpha -1})}^{q}.
\end{multline*}
It remains to choose $\displaystyle A=\left(\frac{C_q^q\|f_{k}\|_{L^{\frac{q}{\beta}}(\mathrm{d}\nu_{\alpha -1})}^{q}}{C_q^q\|f_{k}\|_{L^{\frac{p}{\beta}}(\mathrm{d}\nu_{\alpha -1})}^{p}}\right)^{\frac{1}{q-p}}$ so that 
\[
    \|Tf_{k}\|_{L^{r}(\mathrm{d}\mu)}^{r} \lesssim C_p^{p\frac{q-r}{q-p}}C_q^{q\frac{r-p}{q-p}}\|f_{k}\|_{L^{\frac{p}{\beta}}(\mathrm{d}\nu_{\alpha-1})}^{p\frac{q-r}{q-p}}\|f_{k}\|_{L^{\frac{q}{\beta}}(\mathrm{d}\nu_{\alpha-1})}^{q\frac{r-p}{q-p}}.
\]
Observe that $\theta = \frac{q(r-p)}{r(q-p)}$ and $1-\theta=\frac{p(q-r)}{r(q-p)}$ so that invoking \cref{prop:atom}, we finally obtain  
\begin{align*}
    \|Tf_{k}\|_{L^{r}(\mathrm{d}\mu)}^{r} & \lesssim C_p^{r(1-\theta)}C_q^{r\theta}\lambda_{k}^{p\frac{q-r}{q-p}(\frac{\alpha\beta}{r}-\frac{\alpha\beta}{p})}\|f_{k}\|_{L^{\frac{r}{\beta}}(\mathrm{d}\nu_{\alpha-1})}^{p\frac{q-r}{q-p}}\lambda_{k}^{q\frac{r-p}{q-p}(\frac{\alpha\beta}{r}-\frac{\alpha\beta}{q})}\|f_{k}\|_{L^{\frac{r}{\beta}}(\mathrm{d}\nu_{\alpha-1})}^{q\frac{r-p}{q-p}} \\
    &\lesssim C_p^{r(1-\theta)}C_q^{r\theta}\|f_{k}\|_{L^{\frac{r}{\beta}}(\mathrm{d}\nu_{\alpha-1})}^{r}. \qedhere
\end{align*}
\end{proof}

\subsection{Elementary results on series of lacunary coefficients} 

In the following, assume that $\{\lambda_k\}$ is a lacunary sequence, that an increasing sequence of positive numbers such that $\lambda_{k+1} \geq q \lambda_k$ for some $q >1$. Our first elementary remark is that for any $\kappa >0$ there holds 
\begin{equation}
    \label{eq.lacunary-est}
    \sum_{j=0}^i \lambda_j^{\kappa} \simeq \lambda_i^{\kappa} \quad\text{ and } \sum_{j>i} \lambda_j^{-\kappa} \simeq \lambda_i^{-\kappa}, 
\end{equation}
where the implicit constants only depend on $q$ and $\kappa$. 

\begin{lemma}\label{lem.sequence}
Let $\{\lambda_k\}_{k\geq 0}$ be a lacunary sequence of positive numbers. Let also $\{\varepsilon_k\}_{k\geq 0}$ be a sequence of positive numbers. Let $n\geq 2$ be an integer, $\kappa >0$, $\beta, \tau\in(0,1)$ and some parameter $A\geq 1$. Then: 
\begin{equation}
    \label{eq.sum1}
    \Delta_i^{+}:= \left(\sum_{j>i}\lambda_j^{-\kappa}\varepsilon_j^{\frac{(1-\tau)(1-\beta)}{n-1}}\right)^{n-1} \lesssim \sum_{k_1, \dots, k_{n-1}\geq 0} \prod_{j=1}^{n-1}\delta_{k_j}(i)\lambda_{i+k_jA}^{-\kappa}
\end{equation}
with $\delta_{k}(i)=\left(\displaystyle\sum_{i+kA\leq \ell < i+(k+1)A} \varepsilon_{\ell}\right)^{\frac{(1-\beta)(1-\tau)}{(n-1)}}$. Similarly it holds  
\begin{equation}
    \label{eq.sum2}
    \Delta_i^{-}:=\left(\sum_{j=1}^i\lambda_j^{\kappa}\varepsilon_j^{\frac{(1-\tau)(1-\beta)}{n-1}}\right)^{n-1} \lesssim \sum_{k_1, \dots, k_{n-1}\leq \frac{i}{A}} \prod_{j=1}^{n-1}\delta_{k_j}(i)\lambda_{i-k_jA}^{\kappa},
\end{equation}
with $\delta_{k}(i)=\left(\displaystyle\sum_{i-(k+1)A\leq \ell< i-kA} \varepsilon_{\ell}\right)^{\frac{(1-\beta)(1-\tau)}{(n-1)}}$.
\end{lemma}

\begin{proof} We only present the proof for $\Delta_i^+$ as the treatment of $\Delta_i^{-}$ is similar. We write: 
\[
    \Delta_i^{+} = \left(\sum_{k\geq 0}A_k^{+}(i)\right)^{n-1} = \sum_{k_1, \dots, k_{n-1}\geq 0} \prod_{j=1}^{n-1}A_{k_j}(i),
\]
where
\begin{align*}
    A_{k}(i) &:=\sum_{i+ kA\leq j<i+ (k+1)A}\lambda_{j}^{-\kappa}\varepsilon_{j}^{\frac{(1-\beta)(1-\tau)}{(n-1)}} \\
    &\leq \left(\sum_{i+kA\leq j< i+(k+1)A} \varepsilon_{j}\right)^{\frac{(1-\beta)(1-\tau)}{(n-1)}}\left(\sum_{i+kA\leq j< i+(k+1)A}\lambda_{j}^{-\eta\kappa}\right)^{\frac{1}{\eta}},
\end{align*}
which stems for an application of Hölder's, where $\frac{1}{\alpha} = 1 - \frac{(1-\beta)(1-\tau)}{n-1}$. The conclusion follows from \eqref{eq.lacunary-est}. 
\end{proof}

\section{Proofs for $\beta \geq 1$}

\subsection{Proof of \cref{thm.main-marcinkiewicz}}

We proceed to the proof of \cref{thm.main-marcinkiewicz}, which follows the same lines of the proof developed in \cite{GaLe}, where generalized moments conditions on the measure $\mu$ are used. The corner stone of the proof is a full characterization of the weighted Lebesgue quasi-lacunary M\"{u}ntz spaces given by \cref{thm:Lp-frames}.
We distinguish two cases depending on the value of $r$, and we write $\frac{1}{r}=\frac{1-\vartheta}{p} + \frac{\vartheta}{q}$.

\subsubsection*{Case $r \geqslant 1+\beta$} Let $0<\gamma<1$ and $\theta>0$ be some parameters that will be chosen later. Let $f=\sum\limits_{k\geq 0}f_{k} \in M_{\Lambda}$. It will be convenient to introduce $g_{k}=\displaystyle\frac{\vert Tf_{k} \vert}{\|f_{k}\|_{L^{\frac{r}{\beta}}(\mathrm{d}\nu_{\alpha-1})}}$ if $\|f_{k}\|_{L^{\frac{r}{\beta}}(\mathrm{d}\nu_{\alpha-1})} \neq 0$ and $g_k=0$ otherwise. We write $\lambda_k$ as a short-hand for $\lambda_{n_k}$.
We start by an application of the triangle inequality followed by the Hölder inequality to the series in $k$. We have: 
\begin{multline*}
    \int_{[0,1]}\vert Tf \vert ^{r} \,\mathrm{d}\mu \leq \int_{[0,1]} \left(\sum_{k\geq 0}\vert Tf_{k} \vert \right)^{r} \,\mathrm{d}\mu = \int_{[0,1]} \left(\sum_{k\geq 0}\|f_{k}\|_{L^{\frac{r}{\beta}}(\mathrm{d}\nu_{\alpha-1})} g_k^{\gamma} \lambda_{k}^{-\theta}\times g_k^{1-\gamma} \lambda_{k}^{\theta}\right)^{r}\,\mathrm{d}\mu \\ 
    \leq \int_{[0,1]} \left(\sum_{k\geq 0}\|f_{k}\|_{L^{\frac{r}{\beta}}(\mathrm{d}\nu_{\alpha-1})}^{\frac{r}{\beta}} g_k^{\frac{r\gamma}{\beta}} \lambda_{k}^{-\frac{r\theta}{\beta}}\right)^{\beta}\left(\sum_{k\geq 0} g_k^{\frac{r(1-\gamma)}{r-\beta}}\lambda_{k}^{\frac{r\theta}{r-\beta}}\right)^{r-\beta} \,\mathrm{d}\mu 
\end{multline*}
Since $\beta \geq 1$, an application of the Minkowski inequality with the measure 
\[
    \left(\sum_{k\geq 0} g_k^{\frac{r(1-\gamma)}{r-\beta}}\lambda_{k}^{\frac{r\theta}{r-\beta}}\right)^{r-\beta} \,\mathrm{d}\mu 
\]
gives 
\begin{equation}\label{eq.Minko-1}
    \int_{[0,1]}\vert Tf \vert ^{r} \,\mathrm{d}\mu \leq \left(\sum_{k\geq 0}\|f_{k}\|_{L^{\frac{r}{\beta}}(\mathrm{d}\nu_{\alpha-1})}^{\frac{r}{\beta}} \lambda_{k}^{-\frac{r\theta}{\beta}} \left(\int_{[0,1]} g_{k}^{r\gamma}\left(\sum_{j\geq 0} g_{j}^{\frac{r(1-\gamma)}{r-\beta}}\lambda_{j}^{\frac{r\theta}{r-\beta}}\right)^{r-\beta}\,\mathrm{d}\mu\right)^{\frac{1}{\beta}}\right)^{\beta}.
\end{equation}
As $r-\beta \geq 1$, we apply the Minkowski inequality again to each measure $g_k^{r\gamma}\,\mathrm{d}\mu$, in order to bound  
\begin{align*}
    \int_{[0,1]}\vert Tf \vert ^{r} \,\mathrm{d}\mu &\leq \left(\sum_{k\geq 0}\|f_{k}\|_{L^{\frac{r}{\beta}}(\mathrm{d}\nu_{\alpha-1})}^{\frac{r}{\beta}} \lambda_{k}^{-\frac{r\theta}{\beta}}\left(\sum_{j\geq 0}\lambda_{j}^{\frac{r\theta}{r-\beta}}\left(\int_{[0,1]} g_{k}^{r\gamma}g_{j}^{r(1-\gamma)}\,\mathrm{d}\mu\right)^{\frac{1}{r-\beta}}\right)^{\frac{r-\beta}{\beta}}\right)^{\beta} \\ 
    &=: \left(\sum_{k\geq 0} \|f_{k}\|_{L^{\frac{r}{\beta}}(\mathrm{d}\nu_{\alpha-1})}^{\frac{r}{\beta}} A_k\right)^{\beta},
\end{align*}
where for all $k\geq 0$ we have defined 
\[ 
    A_{k}:=\lambda_{k}^{-\frac{r\theta}{\beta}}\left(\sum_{j\geq 0}\lambda_{j}^{\frac{r\theta}{r-\beta}}\left(\int_{[0,1]} g_{k}^{r\gamma}g_{j}^{r(1-\gamma)}\,\mathrm{d}\mu\right)^{\frac{1}{r-\beta}}\right)^{\frac{r-\beta}{\beta}}.
\]
It is convenient for us to split $A_{k}$ as follows: 
\begin{align}
\label{l1} A_{k} & \lesssim \lambda_{k}^{-\frac{r\theta}{\beta}}\left(\sum_{j=0}^k\lambda_{j}^{\frac{r\theta}{r-\beta}}\left(\int_{[0,1]} g_{k}^{r\gamma}g_{j}^{r(1-\gamma)}\,\mathrm{d}\mu\right)^{\frac{1}{r-\beta}}\right)^{\frac{r-\beta}{\beta}}\\
\label{l2} & + \lambda_{k}^{-\frac{r\theta}{\beta}}\left(\sum_{j> k}\lambda_{j}^{\frac{r\theta}{r-\beta}}\left(\int_{[0,1]} g_{k}^{r\gamma}g_{j}^{r(1-\gamma)}\,\mathrm{d}\mu\right)^{\frac{1}{r-\beta}}\right)^{\frac{r-\beta}{\beta}}.
\end{align}
We start by estimating \eqref{l1}. We use the Hölder inequality as well as \cref{interpolbis} which shows that for all $i\geq 0$,
\[
    \int_{[0,1]}g_i^r\,\mathrm{d}\mu \leqslant C_p^{r(1-\vartheta)}C_q^{r\vartheta}. 
\]
Therefore, we have 
\begin{multline}\label{l7} 
    \eqref{l1} \lesssim \lambda_{k}^{-\frac{r\theta}{\beta}}\left(\sum_{j=0}^{k}\lambda_{j}^{\frac{r\theta}{r-\beta}} \left(\int_{[0,1]}g_{k}^{r}\,\mathrm{d}\mu\right)^{\frac{\gamma}{r-\beta}}\left(\int_{[0,1]}g_{j}^{r}\,\mathrm{d}\mu\right)^{\frac{(1-\gamma)}{r-\beta}}\right)^{\frac{r-\beta}{\beta}} \\
    \lesssim C_p^{\frac{r}{\beta}(1-\vartheta)}C_q^{\frac{r}{\beta}\vartheta}\lambda_{k}^{-\frac{r\theta}{\beta}}\left(\sum_{j=0}^{k}\lambda_{j}^{\frac{r\theta}{r-\beta}}\right)^{\frac{r-\beta}{\beta}} \lesssim C_p^{\frac{r}{\beta}(1-\vartheta)}C_q^{\frac{r}{\beta}\vartheta}\lambda_{k}^{-\frac{r\theta}{\beta}}\lambda_{k}^{\frac{r\theta}{\beta}} \lesssim C_p^{\frac{r}{\beta}(1-\vartheta)}C_q^{\frac{r}{\beta}\vartheta},
\end{multline}
where we have used that the sequence $\{\lambda_k\}_{k\geq 0}$ is lacunary and \eqref{eq.lacunary-est}. 

In order to handle \eqref{l2} we fix $\gamma=\frac{1}{2}$. The general idea would be to follow the previous computations using the Cauchy-Schwarz inequality in the inner integral, but this would slightly fall short of estimating \eqref{l2}. This is the reason why we apply the Hölder inequality with exponents $s_{\varepsilon}:=2(1-\varepsilon)$ and $s'_{\varepsilon} = \frac{2(1-\varepsilon)}{1-2\varepsilon}$ in the inner integral: 
\begin{multline}\label{l3} 
    \sum_{j> k}\lambda_{j}^{\frac{r\theta}{r-\beta}}\left(\int_{[0,1]} g_{k}^{r\gamma}g_{j}^{r(1-\gamma)}\,\mathrm{d}\mu\right)^{\frac{1}{r-\beta}} \\
    \leq \left(\int_{[0,1]} g_{k}^{\frac{rs'_{\varepsilon}}{2}}\,\mathrm{d}\mu\right)^{\frac{1}{s'_{\varepsilon}(r-\beta)}}\sum_{j> k}\lambda_{j}^{\frac{r\theta}{r-\beta}}\left(\int_{[0,1]}g_{j}^{\frac{rs_{\varepsilon}}{2}}\,\mathrm{d}\mu\right)^{\frac{1}{s_{\varepsilon}(r-\beta)}}.
\end{multline}
Since $r \in (p,q)$ we can choose $\varepsilon \ll 1$ such that $p<r(1-\varepsilon)=\frac{rs_{\varepsilon}}{2}<r<\frac{1-\varepsilon}{1-2\varepsilon}=\frac{rs'_{\varepsilon}}{2}<q$. 

Next, we use \cref{interpolbis} with exponent $\frac{rs'_{\varepsilon}}{2} \in (p,q)$, so that 
\begin{equation}\label{l4}
    \|g_{k}\|_{L^{\frac{rs'_{\varepsilon}}{2}}(\mathrm{d}\mu)} = \frac{\|Tf_{k}\|_{L^{\frac{rs'_{\varepsilon}}{2}}(\mathrm{d}\mu)}}{\|f_{k}\|_{L^{\frac{r}{\beta}}(\mathrm{d}\nu_{\alpha-1})}} \leq C_p^{1-\vartheta'_{\varepsilon}}C_q^{\vartheta'_{\varepsilon}}\lambda_{k}^{\alpha \beta (\frac{1}{r}-\frac{2}{rs'_{\varepsilon}})} =  C_p^{1-\vartheta'_{\varepsilon}}C_q^{\vartheta'_{\varepsilon}}\lambda_{k}^{\frac{\alpha \beta \varepsilon}{r(1-\varepsilon)}}, 
\end{equation}
where we have used \cref{prop:atom} and $\frac{2}{rs'_{\varepsilon}}=\frac{1-\vartheta_{\varepsilon}'}{p}+\frac{\vartheta_{\varepsilon}'}{q}$. 
Similarly, an application of \cref{interpolbis} and \cref{prop:atom} gives 
\begin{equation}\label{l4*}
    \|g_j\|_{L^{\frac{rs_{\varepsilon}}{2}}(\mathrm{d}\mu)}=\frac{\|Tf_j\|_{L^{\frac{rs_{\varepsilon}}{2}}(\mathrm{d}\mu)}}{\|f_j\|_{L^{\frac{r}{\beta}}(\mathrm{d}\nu_{\alpha-1}\mathrm{d}x)}} \lesssim C_p^{1-\vartheta_{\varepsilon}}C_q^{\vartheta_{\varepsilon}}\lambda_{j}^{-\frac{\alpha \beta \varepsilon}{r(1-\varepsilon)}},
\end{equation}
with $\frac{2}{rs_{\varepsilon}}=\frac{1-\vartheta_{\varepsilon}}{p}+\frac{\vartheta_{\varepsilon}}{q}$.
Observe that 
\[
    C_p^{1-\vartheta'_{\varepsilon}}C_q^{\vartheta'_{\varepsilon}}C_p^{1-\vartheta_{\varepsilon}}C_q^{\vartheta_{\varepsilon}} = C_p^{1-\vartheta}C_q^{\vartheta}
\]
because $\frac{2-\vartheta'_{\varepsilon}-\vartheta_{\varepsilon}}{p} + \frac{\vartheta_{\varepsilon}'+\vartheta_{\varepsilon}}{q} = \frac{2}{r}\left(\frac{1}{s'_{\varepsilon}}+\frac{1}{s_{\varepsilon}}\right) = \frac{2}{r} = \frac{2(1-\vartheta)}{p}+\frac{2\vartheta}{q}$.
Combining \eqref{l3}, \eqref{l4} and \eqref{l4*} yields 
\begin{equation}\label{l5}
    \eqref{l2} \lesssim C_p^{\frac{r}{\beta}(1-\vartheta)}C_q^{\frac{r}{\beta}\vartheta}\lambda_{k}^{-\frac{r\theta}{\beta}+\frac{\alpha \varepsilon}{2(1-\varepsilon)}}\left(\sum_{j> k}\lambda_{j}^{\frac{r\theta}{r-\beta}-\frac{\alpha \beta \varepsilon }{2(1-\varepsilon)(r-\beta)}}\right)^{\frac{r-\beta}{\beta}}.
\end{equation}

It remains to choose $\theta$ such that $\theta < \frac{\alpha \beta \varepsilon}{2(1-\varepsilon)r}$ in order to ensure convergence of the series. Using that $\{\lambda_k\}_{k\geq 0}$ is lacunary and \eqref{eq.lacunary-est}, we obtain 
\begin{equation}\label{l6} 
    \eqref{l2} \lesssim C_p^{\frac{r}{\beta}(1-\vartheta)}C_q^{\frac{r}{\beta}\vartheta}\lambda_{k}^{-\frac{r\theta}{\beta}+\frac{\alpha \varepsilon}{2(1-\varepsilon)}}\times \lambda_{k}^{\frac{r\theta}{\beta}-\frac{\alpha\varepsilon}{2(1-\varepsilon)}} \lesssim C_p^{\frac{r}{\beta}(1-\vartheta)}C_q^{\frac{r}{\beta}\vartheta}.
\end{equation}
Combining \eqref{l7} and \eqref{l6} we arrive at $A_k \lesssim C_p^{\frac{r}{\beta}(1-\vartheta)}C_q^{\frac{r}{\beta}\vartheta}$ uniformly in $k$, and the conclusion follows from \cref{thm:Lp-frames}:
\[
    \int_{[0,1]}\vert Tf \vert ^{r} \,\mathrm{d}\mu \lesssim C_p^{r(1-\vartheta)}C_q^{r\vartheta}\left(\sum_{k\geq 0}\|f_{k}\|_{L^{\frac{r}{\beta}}(\mathrm{d}\nu_{\alpha-1})}^{\frac{r}{\beta}}\right)^{\beta}\lesssim C_p^{r(1-\vartheta)}C_q^{r\vartheta}\|f\|_{L^{\frac{r}{\beta}}(\mathrm{d}\nu_{\alpha-1})}^{r}.
\]

\subsubsection*{Case $\beta < r < 1+\beta$} Let $0<\gamma<1$ and $\theta>0$ be some parameters that will be chosen suitably in following. As in the first case, let $f= \sum_{k\geq 0} f_k \in M_{\Lambda}$. We can start from \eqref{eq.Minko-1}. Our task is therefore to estimate the  
\[
    B_k = \lambda_{k}^{-\frac{r\theta}{\beta}}\left(\int_{[0,1]} g_{k}^{r\gamma}\left(\sum_{j\geq 0} g_{j}^{\frac{r(1-\gamma)}{r-\beta}}\lambda_{j}^{\frac{r\theta}{r-\beta}}\right)^{r-\beta}\,\mathrm{d}\mu\right)^{\frac{1}{\beta}},
\]
which we split as 
\begin{align}\label{eql1} B_k &= \lambda_{k}^{-\frac{r\theta}{\beta}}\left(\int_{[0,1]} g_{k}^{r\gamma}\left(\sum_{j=0}^{k} g_{j}^{\frac{r(1-\gamma)}{r-\beta}}\lambda_{j}^{\frac{r\theta}{r-\beta}}\right)^{r-\beta}\,\mathrm{d}\mu\right)^{\frac{1}{\beta}} \\
    &\label{eql2} + \lambda_{k}^{-\frac{r\theta}{\beta}}\left(\int_{[0,1]} g_{k}^{r\gamma}\left(\sum_{j>k} g_{j}^{\frac{r(1-\gamma)}{r-\beta}}\lambda_{j}^{\frac{r\theta}{r-\beta}}\right)^{r-\beta}\,\mathrm{d}\mu\right)^{\frac{1}{\beta}}.
\end{align}

Using the same argument as in the previous case, the proof boils down to proving $B_k \lesssim C_p^{r(1-\vartheta)}C_q^{r\vartheta}$, which is done by estimating separately \eqref{eql1} and \eqref{eql2}. 

We start estimating \eqref{eql2} with an application of the H\"{o}lder inequality in the inner integral, with exponents $\frac{1}{r-\beta}$ and $\frac{1}{1+\beta-r}$ (since $r-\beta >0$), which gives 
\[
    \eqref{eql2} \lesssim \lambda_k^{-\frac{r\theta}{\beta}}\left(\int_{[0,1]} g_{k}^{\frac{r\gamma}{1+\beta-r}}\,\mathrm{d}\mu\right)^{\frac{1+\beta-r}{\beta}}\left(\sum_{j> k}\lambda_{j}^{\frac{r\theta}{r-\beta}}\int_{[0,1]} g_{j}^{\frac{r(1-\gamma)}{r-\beta}}\,\mathrm{d}\mu \right)^{\frac{r-\beta}{\beta}}.
\]
The next step is an application of \cref{interpolbis} and \cref{prop:atom}. In order to do so, let us make clear that we can choose $\gamma$ so that 
\begin{equation}\label{eq.cond1}
    p<\frac{r(1-\gamma)}{r-\beta}<r<\frac{r\gamma}{1+\beta-r}<q.
\end{equation}
This is indeed possible because $p<r$ and $r<q$, which implies $(1+\beta-r)<1-\frac{p}{r}(r-\beta)$ and $(1+\beta-r)<\frac{q}{r}(1+\beta-r)$. It then suffices to choose 
\[
    (1+\beta-r)<\gamma<\min\left\{1-\frac{p}{r}(r-\beta),\frac{q}{r}(1+\beta-r)\right\}
\]
in order to ensure \eqref{eq.cond1}. Then, \cref{interpolbis} followed by \cref{prop:atom} imply 
\begin{align}
    \left(\int_{[0,1]} g_{k}^{\frac{r\gamma}{1+\beta-r}}\,\mathrm{d}\mu\right)^{\frac{1+\beta-r}{\beta}} &\lesssim C_p^{\frac{r\gamma}{\beta}(1-\vartheta_1)}C_q^{\frac{r\gamma}{\beta}\vartheta_1}\lambda_{k}^{\frac{r\gamma}{1+\beta-r} \frac{1+\beta-r}{\beta}  \alpha \beta (\frac{1}{r}-\frac{1+\beta-r}{r\gamma})}\notag \\
    &\lesssim C_p^{\frac{r\gamma}{\beta}(1-\vartheta_1)}C_q^{\frac{r\gamma}{\beta}\vartheta_1}\lambda_{k}^{\alpha(r+\gamma-\beta-1)},\label{eql3}  
\end{align}
where $\frac{1+\beta-r}{r\gamma}=\frac{1-\vartheta_1}{p}+\frac{\vartheta_1}{q}$. Also, 
\begin{align}
    \int_{[0,1]} g_{j}^{\frac{r(1-\gamma)}{r-\beta}}\,\mathrm{d}\mu &\lesssim C_p^{\frac{r(1-\gamma)}{r-\beta}(1-\vartheta_2)}C_q^{\frac{r(1-\gamma)}{r-\beta}\vartheta_2}\lambda_{j}^{\frac{r(1-\gamma)}{r-\beta} \alpha \beta (\frac{1}{r}-\frac{r-\beta}{r(1-\gamma)})} \notag\\ 
    &\lesssim C_p^{\frac{r(1-\gamma)}{r-\beta}(1-\vartheta_2)}C_q^{\frac{r(1-\gamma)}{r-\beta}\vartheta_2}\lambda_{j}^{-\frac{\alpha \beta (r+\gamma-\beta-1)}{r-\beta}}. \label{eqm1} 
\end{align}
where $\frac{r-\beta}{r(1-\gamma)}=\frac{1-\vartheta_2}{p}+\frac{\vartheta_2}{q}$.  

Therefore, choosing $\theta$ such that $0<\theta <\frac{\alpha \beta (r+\gamma-\beta-1)}{r}$, which is possible since $r+\gamma - \beta >0$, it follows that  
\begin{align}
    \left(\sum_{j>k}\lambda_{j}^{\frac{r\theta}{r-\beta}}\int_{[0,1]}g_{j}^{\frac{r(1-\gamma)}{r-\beta}}\,\mathrm{d}\mu \right)^{\frac{r-\beta}{\beta}} &\lesssim C_p^{\frac{r(1-\gamma)}{\beta}(1-\vartheta_2)}C_q^{\frac{r(1-\gamma)}{\beta}\vartheta_2}\left(\sum_{j\geq k} \lambda_{j}^{\frac{r\theta}{r-\beta}}\lambda_{j}^{\frac{-\alpha \beta (r+\gamma-\beta-1)}{r-\beta}}\right)^{\frac{r-\beta}{\beta}}\notag \\
    &\lesssim C_p^{\frac{r(1-\gamma)}{\beta}(1-\vartheta_2)}C_q^{\frac{r(1-\gamma)}{\beta}\vartheta_2}\lambda_{k}^{-\alpha(r+\gamma-\beta-1)+\frac{r\theta}{\beta}},\label{eql4}
\end{align}
thanks to the choice of $\theta$ and because $\{\lambda_{k}\}_{k\geq 0}$ is lacunary.
Combining \eqref{eql3} and \eqref{eql4} we get that
\begin{equation}\label{eql5}
\eqref{eql2} \lesssim C_p^{\frac{r}{\beta}(1-\vartheta)}C_q^{\frac{r}{\beta}\vartheta}\lambda_{k}^{-\frac{r\theta}{\beta}}\lambda_{k}^{\alpha(r+\gamma-\beta-1)}\lambda_{k}^{-\alpha (\gamma+r-\beta-1)+\frac{r\theta}{\beta}}\lesssim C_p^{\frac{r}{\beta}(1-\vartheta)}C_q^{\frac{r}{\beta}\vartheta}.
\end{equation}

To estimate \eqref{eql1} we use the same choice of $\gamma$, and we also apply the H\"{o}lder inequality with exponents $\frac{1}{\gamma}$ and $\frac{1}{1-\gamma}$ in the inner integral:
\begin{multline}\label{eqm3}
    \eqref{eql1} \lesssim \lambda_k^{-\frac{r\theta}{\beta}}\left(\int_{[0,1]} g_{k}^{r} \,\mathrm{d}\mu \right)^{\frac{\gamma}{\beta}}\left(\int_{[0,1]}\left(\sum_{j=0}^{k} g_{j}^{\frac{r(1-\gamma)}{r-\beta}}\lambda_{j}^{\frac{r\theta}{r-\beta}}\right)^{\frac{r-\beta}{1-\gamma}}\,\mathrm{d}\mu\right)^{\frac{1-\gamma}{\beta}}\\
    \lesssim C_p^{\frac{r\gamma}{\beta}(1-\vartheta)}C_q^{\frac{r\gamma}{\beta}\vartheta}\lambda_k^{-\frac{r\theta}{\beta}}\left(\int_{[0,1]} \left(\sum_{j=1}^{k} g_{j}^{\frac{r(1-\gamma)}{r-\beta}}\lambda_{j}^{\frac{r\theta}{r-\beta}}\right)^{\frac{r-\beta}{1-\gamma}}\,\mathrm{d}\mu\right)^{\frac{1-\gamma}{\beta}},
\end{multline}
where in the last inequality we have used \cref{interpolbis}. Next, as $\frac{r-\beta}{1-\gamma}>1$, the Minkowski inequality yields 
\begin{multline}\label{eqm4}
    \eqref{eql1} \lesssim C_p^{\frac{r\gamma}{\beta}(1-\vartheta)}C_q^{\frac{r\gamma}{\beta}\vartheta}\lambda_k^{-\frac{r\theta}{\beta}} \left(\sum_{j=0}^{k}\lambda_{j}^{\frac{r\theta}{r-\beta}}\left(\int_{[0,1]}g_{j}^{r}\,\mathrm{d}\mu\right)^{\frac{1-\gamma}{r-\beta}}\right)^{\frac{r-\beta}{\beta}}\\
    \lesssim C_p^{\frac{r}{\beta}(1-\vartheta)}C_q^{\frac{r}{\beta}\vartheta}\lambda_k^{-\frac{r\theta}{\beta}}\left(\sum_{j=0}^{k}\lambda_{j}^{\frac{r\theta}{r-\beta}}\right)^{\frac{r-\beta}{\beta}} \lesssim C_p^{\frac{r}{\beta}(1-\vartheta)}C_q^{\frac{r}{\beta}\vartheta},
\end{multline}
where we have used \cref{interpolbis} and because $\{\lambda_{k}\}_{k\geq 0}$ is lacunary (see \eqref{eq.lacunary-est}). 
 
Combining \eqref{eql5} and \eqref{eqm4} yields $B_k \lesssim C_p^{r(1-\theta)}C_q^{r\theta}$ and therefore ends the proof.

Let us prove that under the assumption that one of the sequence $\{C_p(k)\}_{k\geq 0}$ or $\{C_q(k)\}_{k\geq 0}$ is bounded by a sequence $\varepsilon_k$ that vanishes as $k\to \infty$, then $T$ is compact. Let us introduce $T_n:=TS_n$, where $S_nf:=\sum_{k=0}^n f_k$. The above arguments can be applied to the operator $T-T_n$, which is then of strong type $r$ with norm bounded by $\sup_{k>n}C_p(k)^{1-\vartheta}C_q(k)^{\vartheta}\to 0$ as $n\to \infty$, hence $T$ is compact. 

\subsection{Consequences of \cref{thm.main-marcinkiewicz}} 

In this section we prove the three corollaries to \cref{thm.main-marcinkiewicz} stated in \cref{sec.consequence-subcritical}.

\begin{proof}[Proof of \cref{coro:boundedness}]
We start with an application of \cref{prop:atom}, which yields 
\[ 
    \|f_{k}\|_{L^{\infty}} \leq C\lambda_{n_k}^{\frac{\alpha \beta}{r}}\|f_{k}\|_{L^{\frac{r}{\beta}}(\mathrm{d}\nu_{\alpha - 1})} := \Lambda_{k}.
\] 
We can then use this information in the Cavalieri formula: 
\[
    \int_{[0,1]}\vert f_k \vert ^{r} \,\mathrm{d}\mu  = r\int_{0}^{+\infty} \lambda^{r-1}\mu(\vert f_{k} \vert > \lambda) \,\mathrm{d}\lambda = r\int_{0}^{\Lambda_k} \lambda^{r-1}\mu(\vert f_{k} \vert > \lambda)\,\mathrm{d}\lambda, 
\] 
since by assumption $\imath_{\mu}$ is of restricted weak-type $p$, we infer 
\begin{equation*}
    \int_{[0,1]}\vert f_k \vert ^{r} \,\mathrm{d}\mu \lesssim \|f_{k}\|_{L^{\frac{p}{\beta}}(\mathrm{d}\nu_{\alpha-1})}^{p}\int_{0}^{\Lambda_{k}}\lambda^{r-p-1}\,\mathrm{d}\lambda \lesssim \|f_{k}\|_{L^{\frac{p}{\beta}}(\mathrm{d}\nu_{\alpha - 1})}^{p}\lambda_{n_k}^{\alpha \beta \frac{r-p}{p}}\|f_{k}\|_{L^{\frac{r}{\beta}}(\mathrm{d}\nu_{\alpha - 1})}^{r-p},
\end{equation*}
where in the last line we have used $r-p-1 >-1$.
A further application of \cref{prop:atom} finally yields 
\[
    \int_{[0,1]}\vert f_k \vert ^{r} \,\mathrm{d}\mu \lesssim  \lambda_{n_k}^{\alpha \beta \frac{r-p}{p}}\lambda_{n_k}^{r\alpha \beta(\frac{1}{r}-\frac{1}{p})}\|f_{k}\|_{L^{\frac{r}{\beta}}(\mathrm{d}\nu_{\alpha - 1})}^{r}
    = \|f_{k}\|_{L^{\frac{r}{\beta}}(\mathrm{d}\nu_{\alpha - 1})}^{r}.
\] 
We thus have obtained that $\imath_{\mu}$ is of restricted strong type for any $r>p$, from which the conclusions now follows from \cref{thm.main-marcinkiewicz}. 
\end{proof}

\begin{proof}[Proof of \cref{coro.embeddings-subcritical}]
Within this proof, we slightly abuse notation by writing $\lambda_{k}$ instead of $\lambda_{n_k}$. Therefore, assumption $B_{p,\Lambda}(\alpha,\beta)$ writes 
\begin{equation}
\label{eq.BpL-use1}  
    \int_{[0,1]}t^{p\lambda_{k}}\mathrm{d}\mu(t)\leq C_k\lambda_{k}^{-\alpha\beta}
\end{equation} 
for some $C_k\leq C <\infty$.

Assume that $\ell\geq \max\{\beta,pN\}$. Then an application of \cref{thm:muntz-bound} followed by \cref{prop:atom} yields
\begin{equation}
    \label{eq.Bp-use}
    \int_{[0,1]}\vert f_{k}(t) \vert ^{\ell} \,\mathrm{d}\mu(t)\lesssim \|f_{k}\|_{L^{\infty}}^{\ell}\int_{[0,1]} t^{\frac{\ell\lambda_{k}}{N}}\,\mathrm{d}\mu(t)
    \lesssim \|f_{k}\|_{L^{\frac{\ell}{\beta}}(\mathrm{d}\nu_{\alpha - 1})}^{\ell}\lambda_{k}^{\alpha\beta}\int_{[0,1]} t^{p\lambda_{k}}\,\mathrm{d}\mu(t)
\end{equation}
where we have used \cref{prop:atom} and because $\ell \geq pN$. Now an application of \eqref{eq.BpL-use1} yields 
\begin{equation}
    \label{eq.bounded}
    \int_{[0,1]}\vert f_{k}(t) \vert ^{\ell} \,\mathrm{d}\mu(t) \leq C_k\|f_{k}\|_{L^{\frac{\ell}{\beta}}(\mathrm{d}\nu_{\alpha - 1})}^{\ell},
\end{equation}
or equivalently, that $\imath_{\mu}$ is of restricted strong-type $\ell$. This holds for any $\ell\geq \max\{\beta,pN\}$. 

Assume now that $pN\leq \beta$. Then $\imath_{\mu}$ is of restricted strong type $r$ for all $r\geq \beta$, so that \cref{thm.main-marcinkiewicz} proves that $\imath$ is of strong type $r$ for any $r>\beta$. The case $r=\beta$ follows from \cref{rem:strong}.
    
Assume now that $pN>\beta$. Then we know that $\imath_{\mu}$ is of restricted strong type $\ell$ for any $\ell>pN$, from wich we infer from \cref{thm.main-marcinkiewicz} that $\imath_{\mu}$ is of strong type $r$, for any $r>pN$.
\end{proof}  

\begin{remark}\label{rem.compact-sub}
The compactness statement can be easily obtained as follows: let $S_{n}$ defined as $S_nf=\sum_{k=0}^nf_k$ for any M{\"u}ntz polynomial $f= \sum_{k\geq 0} f_k$. Then $T_n:=\imath_{\mu}S_n$ is a finite rank operator and one can use \eqref{eq.bounded} to obtain that $\imath_{\mu} - T_n$ is of restricted weak type $\ell$, for any $\ell \geq \{\beta, pN\}$ and with constant $C_{\ell}(k)$, therefore \cref{thm.main-marcinkiewicz} implies that $\imath_{\mu} - T_n$ is of strong type $r$ with constant at most $s\sup_{k>n}C_r(k) \to 0$ as $n\to \infty$. This proves that $\imath$ is compact.    
\end{remark}

\begin{proof}[Proof of \cref{coro.subcritical1}] Again in this proof we write $\lambda_k$ in place of $\lambda_{n_k}$. 
In order to prove (\textit{i}) $\Longrightarrow$ (\textit{ii}) and by virtue of \cref{coro.embeddings-subcritical} it is enough to check that a measure $\mu \in \mathcal{M}_{x^{\alpha\beta}}$ satisfies the assumption $B_{\varepsilon,\Lambda}(\alpha,\beta)$ for any $\varepsilon>0$. We recall \cite[Lemma 2.2]{CFT} that since $\mu([1-\varepsilon])\leqslant \varepsilon^{\alpha\beta} =:\rho(\varepsilon) $, there holds 
\[
    \int_{[0,1]} t^{\varepsilon\lambda_k} \mathrm{d}\mu(t) \lesssim \int_{[0,1]} t^{\varepsilon\lambda_k} \rho'(1-t)\,\mathrm{d}t \lesssim  \int_{[0,1]} t^{\varepsilon\lambda_k}(1-t)^{\alpha\beta - 1}\,\mathrm{d}t \simeq \lambda_k^{-\alpha\beta}.
\]
    
That (\textit{ii}) $\Longrightarrow$ (\textit{i}) can be seen from adapting the argument in \cite[Section 3]{LM24}. For the reader's convenience we reproduce it: first as $(1-\varepsilon)^{\frac{p}{\varepsilon}} \underset{\varepsilon \to 0}{\longrightarrow} e^{-p}$, we have 
\[
	\mu([1-\varepsilon, 1])=\int_{1-\varepsilon}^1 \,\mathrm{d}\mu \leqslant Ce^p \int_{1-\varepsilon}^1 (1-\varepsilon)^{\frac{p}{\varepsilon}} \,\mathrm{d}\mu \lesssim \int_0^1 t^{\frac{p}{\varepsilon}}\,\mathrm{d}\mu, 
\]
Now, with $\varepsilon = \lambda_k^{-1}$ and applying (\textit{ii}) to the monomial $f(t)=t^{\lambda_{k}}$ leads to
\[
	\mu([1-\lambda_k^{-1}, 1]) \lesssim \int_{[0,1]}t^{\lambda_{k}p}\,\mathrm{d}\mu(t) \lesssim \|t^{\lambda_k}\|_{L^{\frac{p}{\beta}}(\mathrm{d}\nu_{\alpha - 1})}^{\beta} \lesssim \lambda_k^{-\alpha\beta}. 	
\]
The conclusion follows from the positivity of $\mu$. Indeed, for any $\varepsilon >0$ we chose $k\geqslant 0$ such that $\lambda_{k+1}^{-1}\leq \varepsilon \leq \lambda_k^{-1}$ and conclude 
\[
	\mu([1-\varepsilon,1]) \leqslant \mu([1-\lambda_k^{-1}, 1]) \lesssim \lambda_k^{-\alpha\beta} \lesssim \varepsilon^{\alpha\beta},
\]
which proves that $\mu \in\mathcal{M}_{x^{\alpha\beta}}$. Note that in the last line we have used that $\{\lambda_k\}_{k\geq 0}$ is subgeometric. 
\end{proof}

\subsection{An example of application}\label{sec.example-subcritical}

Let us first illustrate how one can apply \cref{thm.main-marcinkiewicz} on a non-trivial example. We assume that $\Lambda$ is a lacunary sequence. Let $\{\mu_n\}_{n\geq 0}$ and $\{c_n\}_{n\geq 0}$ be sequences of positive real numbers, that we are going to choose later on, and consider the operator $T$ defined by
\[
    Tf(t) = \sum_{n\geq 1} c_nf(t^{\mu_n}), 
\]
and let $\mu$ be the measure $(1-x)^{\gamma - 1}\mathrm{d}x$, where $\gamma >0$ is to be chosen later. 
This operator should be viewed as a multiplicative version of the operators considered in \cite{BW09}. 

As $T(t^{\lambda_k})=\sum_{n\geq 1} c_nt^{\lambda_k\mu_n}$, an application of \cref{thm:Lp-frames} gives  
\begin{equation}
    \label{res-strong-web}
    \|T(t^{\lambda_k})\|_{L^p(\mathrm{d}\nu_{\gamma - 1})}^p \simeq \sum_{n\geq 0} |c_n|^p \|t^{\lambda_k\mu_n}\|_{L^p(\mathrm{d}\nu_{\gamma-1})}^p \simeq \lambda_k^{-\gamma}\sum_{n\geq 1} \frac{|c_n|^p}{\mu_n^{\gamma}}. 
\end{equation}
Assume that the $c_n, \mu_n$ are chosen in a way that $C:=\sum_{n\geq 1} \frac{|c_n|^p}{\mu_n^{\gamma}}<\infty$. Then, since $\|t^{\lambda_k}\|_{L^{\frac{p}{\beta}}(\mathrm{d}\nu_{\alpha - 1})} \simeq \lambda_k^{-\frac{\alpha\beta}{p}}$, we see that by imposing $\gamma=\alpha\beta >0$ we can infer from \eqref{res-strong-web} that $T$ is of restricted strong type $p$ for any $p\geq 1$, therefore of strong type $r$ for any $r\geq 1$ by \cref{thm.main-marcinkiewicz}. 

Observe that this cannot be trivially obtained, as one would for instance apply the triangle inequality for $r>\beta$ and the strong type $r$ assumption, with $f(t) = \sum_{k\geq 0} a_kt^{\lambda_k}$, 
\[
    \|Tf\|_{L^r(\mathrm{d}\nu_{\gamma - 1})} \leqslant \sum_{k\geq 0} |a_k|\|T(t^{\lambda_k})\|_{L^r(\mathrm{d}\nu_{\gamma - 1})} \leqslant C \sum_{k\geq 0} |a_k|\|t^{\lambda_k}\|_{L^{\frac{r}{\beta}}(\mathrm{d}\nu_{\alpha})}. 
\]
Now, an application of \cref{thm:Lp-frames} yields 
\[
    \|f\|_{L^{\frac{r}{\beta}}(\mathrm{d}\nu_{\alpha})} \simeq \left(\sum_{k\geq 0} |a_k|^{\frac{r}{\beta}}\|t^{\lambda_k}\|^{\frac{r}{\beta}}_{L^{\frac{r}{\beta}}(\mathrm{d}\nu_{\alpha})}\right)^{\frac{\beta}{r}}, 
\]
from which the Hölder inequality implies 
\[
    \|Tf\|_{L^r(\mathrm{d}\nu_{\gamma - 1})} \leqslant C\left(\sum_{k\geq 1}1\right)^{\frac{r-\beta}{\beta}} \|f\|_{L^{\frac{r}{\beta}}(\mathrm{d}\nu_{\alpha})} = +\infty,  
\]
therefore not yielding the $L^r(\mathrm{d}\mu)$ estimate. 

\subsection{A counter-example}\label{sec.C-example-subcritical}

Let $\{\lambda_k\}_{k\geq 0}$ be a lacunary sequence, and $\beta \geq 1$. 
Our goal is to construct an operator $T$ defined by its action on $t^{\lambda_k}$ by 
\begin{equation}
    \label{def.T-kernel}
    T(t^{\lambda_k}) = \sum_{n\geq 0} c_n(k)t^{\lambda_n},
\end{equation}
where we are to choose the $c_n(k)$, and such that $T$ is of restricted strong type $r>\beta$ but not of strong type $r$. To this end, let $\alpha, \gamma >0$, $\varepsilon >0$ (to be chosen small enough in the proof) and choose $\mathrm{d}\mu=\mathrm{d}\nu_{\gamma - 1}$. Let also $N\geq 1$ be an integer. 

We consider
\[
    f_N(t):=\sum_{k=1}^N \lambda_k^{\frac{\beta\alpha}{r}}t^{\lambda_k}=\sum_{k=1}^N a_kt^{\lambda_k},
\]
which in virtue of \cref{thm:Lp-frames}, is such that
\begin{equation}
    \label{eq.fN}
    \|f_N\|_{L^{\frac{r}{\beta}}(\mathrm{d}\nu_{\gamma})} \simeq \left(\sum_{k=1}^N a_k^{\frac{r}{\beta}}\lambda_k^{-\alpha}\right)^{\frac{\beta}{r}} = \left(\sum_{k=1}^N1\right)^{\frac{\beta}{r}} \simeq N^{\frac{\beta}{r}}.
\end{equation}
We define $c_n(k)=\lambda_k^{-\frac{\beta\alpha}{r}}\lambda_n^{\frac{\gamma}{r}}\frac{\mathbf{1}_{1\leq n \leq k}}{n^{\frac{1+\varepsilon}{r}}}$. 
Recalling that $\|t^{\lambda_n}\|_{L^{r}(\mathrm{d}\nu_{\gamma-1})} \simeq \lambda_n^{-\frac{\gamma}{r}}$, we can compute by \cref{thm:Lp-frames}, 
\[
    \|T(t^{\lambda_k})\|_{L^r(\mathrm{d}\nu_{\gamma-1})} \simeq \left(\sum_{n\geq 0} |c_n(k)|^r \|t^{\lambda_n}\|^r_{L^{r}(\mathrm{d}\nu_{\gamma - 1})}\right)^{\frac{1}{r}} \lesssim  \lambda_k^{-\frac{\alpha\beta}{r}} \left(\sum_{n=1}^k\frac{1}{n^{1+\varepsilon}}\right)^{\frac{1}{r}}\lesssim \|t^{\lambda_k}\|_{L^{\frac{r}{\beta}}(\mathrm{d}\nu_{\alpha-1})},
\]
because $\|t^{\lambda_k}\|_{L^{\frac{r}{\beta}}(\mathrm{d}\nu_{\alpha-1})} \simeq \lambda_k^{-\frac{\alpha\beta}{r}}$. Therefore we conclude that $T$ defined in \eqref{def.T-kernel} is of restricted strong type $r$. 

However, writing $Tf_N = \sum_{n\geq 0} \left(\sum_{k=1}^Na_kc_n(k) \right) t^{\lambda_n}$ we infer:
\[
    \|Tf_N\|_{L^r(\mathrm{d}\nu_{\gamma-1})} \simeq \left(\sum_{n\geq 0}\left\vert\sum_{k=1}^Na_kc_n(k) \right\vert^r\|t^{\lambda_n}\|_{L^r(\mathrm{d}\nu_{\gamma - 1})}^r\right)^{\frac{1}{r}} \simeq \left(\sum_{n=1}^N\frac{1}{n^{1+\varepsilon}}\left(\sum_{k=n}^N1\right)^r\right)^{\frac{1}{r}}, 
\]
so that we can lower bound 
\[
    \|Tf_N\|_{L^r(\mathrm{d}\nu_{\gamma-1})} \gtrsim \left(\sum_{\frac{N}{3}\leq n\leq 
    \frac{2N}{3}}\frac{(N-n)^r}{n^{1+\varepsilon}}\right)^{\frac{1}{r}}\gtrsim N^{1-\frac{\varepsilon}{r}}
\]
which in conjunction with \eqref{eq.fN} and taking $N$ large enough as soon as $r>\beta$ and $\varepsilon \ll 1$ proves that $T$ is not of strong type $r$.

\section{Proofs for $0< \beta <1$}

\subsection{Proof of \cref{thm.interpol-supercritical}}

Let us first explain how the compactness of $T$ follows easily from the proof that $T$ is of strong type $r$. Consider the finite rank operator $T_n=TS_n$ where $S_n$ is defined for any $f(t)=\sum_{k\geq 0} a_kt^{\lambda_k}$ by $S_nf(t)=\sum_{k=0}^n a_kt^{\lambda_k}$. Then \cref{thm.interpol-supercritical} applied to $T-T_n$ proves that $T$ is of strong type $r$ with norm bounded by 
$\left(\sum_{k>n}\varepsilon_k\right)C_p^{1-\theta}C_q^{\theta}$, which goes to $0$ as $n\to \infty$. 

We are now going to prove the theorem, first in the integer case, and then in the non integer case. In the following computations we are going to prove that for any $f \in M_{\Lambda}$ there holds 
\[
    \|Tf\|_{L^r(\mathrm{d}\mu)} \lesssim \left(\sum_{k\geq 0}\varepsilon_k\right)\|f\|_{L^{\frac{r}{\beta}}(\mathrm{d}\nu_{\alpha - 1})}, 
\] 
and we are not going to track the dependance on $C_p$ and $C_q$. One can check as in the proof of \cref{thm.main-marcinkiewicz} that the implicit constant in the above is $C_p^{1-\theta}C_q^{\theta}$.

In this section we write $\lambda_k$ in place of $\lambda_{n_k}$ for notational convenience. 

\subsubsection{Integer case $r=n\geq 1$.}

Let us first deal with the case $n=1$, for which we simply use the triangle inequality, the assumption and Hölder's inequality: 
\begin{multline}\label{eq.m1}
    \|Tf\|_{L^1(\mathrm{d}\mu)} \leq \sum_{k\geq 0} \|Tf_k\|_{L^1(\mathrm{d}\mu)} \lesssim \sum_{k\geq 0} \varepsilon_k^{1-\beta}\|f_k\|_{L^{\frac{1}{\beta}}(\mathrm{d}\nu_{\alpha - 1})} \\ 
    \lesssim \left(\sum_{k\geq 0} \varepsilon_k \right)^{1-\beta}\left(\sum_{k\geq 0} \|f_k\|^{\frac{1}{\beta}}_{L^{\frac{1}{\beta}}(\mathrm{d}\nu_{\alpha - 1})}\right)^{\beta}, 
\end{multline}
and the conclusion follows from \cref{thm:Lp-frames}.

We assume now $n\geq 2$ and start by applying the triangle inequality:
\begin{align*}
    \|Tf\|_{L^n(\mathrm{d}\mu)}^n &\lesssim \int_{[0,1]} \left(\sum_{k \geq 0} |Tf_{k}(t)|\right)^n \,\mathrm{d}\mu(t) \\ 
    &= \int_{[0,1]} \left(\sum_{k \geq 0} \|f_{k}\|_{L^{\frac{n}{\beta}}(\mathrm{d}\nu_{\alpha-1})} A_{k}^{\gamma}(t)\lambda_{k}^{-\theta}A_{k}^{1-\gamma}(t)\lambda_{k}^{\theta}\right)^n \,\mathrm{d}\mu, 
\end{align*}
where $A_{k}(t) = \frac{|Tf_{k}(t)|}{\|f_{k}\|_{L^{\frac{n}{\beta}}(\mathrm{d}\nu_{\alpha-1})}}$ if $\|f_{k}\|_{L^{\frac{n}{\beta}}(\mathrm{d}\nu_{\alpha-1})} \neq 0$ and $A_k(t)=0$ otherwise, and $\theta >0$, $\gamma \in (0,1)$ are to be chosen in the proof. 
We will use the following bound on the norm of $A_k$ (for any $s\in (p,q)$),
\begin{equation}
    \label{eq.Ak-bound}
    \|A_{k}\|_{L^s(\mathrm{d}\mu)} \lesssim \varepsilon_k^{\frac{1-\beta}{s}} \lambda_{k}^{\alpha\beta(\frac{1}{n}-\frac{1}{s})}.
\end{equation}
The proof of \eqref{eq.Ak-bound} follows from \cref{interpolbis} applied with $\mathrm{d}\mu=\mathrm{d}\nu_{\alpha - 1}$ followed by \cref{prop:atom},  
\[
    \|A_{k}\|_{L^s(\mathrm{d}\mu)} \lesssim \frac{\|Tf_{k}\|_{L^s(\mathrm{d}\mu)}}{\|f_{k}\|_{L^{\frac{n}{\beta}}(\mathrm{d}\nu_{\alpha - 1})}} \lesssim \varepsilon_k^{\frac{1-\beta}{s}}\frac{\|f_{k}\|_{L^{\frac{s}{\beta}}(1-x)^{\alpha-1}\mathrm{d}x}}{\|f_{k}\|_{L^{\frac{n}{\beta}}(\mathrm{d}\nu_{\alpha - 1})}}. 
\]
The Hölder inequality $\ell^n \times \ell^{\frac{n}{n-1}}$ in the inner summation yields  
\[
    \|Tf\|_{L^n(\mathrm{d}\mu)}^n \lesssim \sum_{i \geq 0} \|f_{i}\|^n_{L^{\frac{n}{\beta}}(\mathrm{d}\nu_{\alpha - 1})}\lambda_i^{-n\theta}\int_{[0,1]} A_i^{\gamma n}(t)\left(\sum_{j \geq 0}\lambda_j^{\frac{n\theta}{n-1}} A_j^{{\frac{(1-\gamma)n}{n-1}}}(t)\right)^{n-1}\mathrm{d}\mu(t).
\] 
An application of Minkowski's inequality to the measure $A_i^{\gamma n}(t)\mathrm{d}\mu(t)$ now yields 
\[
    \|Tf\|_{L^n(\mathrm{d}\mu)}^n \lesssim \sum_{i \geq 0} \|f_{i}\|^n_{L^{\frac{n}{\beta}}(\mathrm{d}\nu_{\alpha-1})}\lambda_i^{-n\theta} \left(\sum_{j \geq 0}\lambda_j^{\frac{n\theta}{n-1}} \left(\int_{[0,1]}A_i^{\gamma n}(t)A_j^{(1-\gamma)n}(t)\mathrm{d}\mu(t)\right)^{\frac{1}{n-1}}\right)^{n-1}.
\]
Let us introduce the following quantity: 
\[
    C_i := \lambda_i^{-n\theta}\left(\sum_{j \geq 0}\lambda_j^{\frac{n\theta}{n-1}} \left(\int_{[0,1]}A_i^{\gamma n}(t)A_j^{(1-\gamma)n}(t)\mathrm{d}\mu(t)\right)^{\frac{1}{n-1}}\right)^{n-1}.
\]
Note that if we were able to prove the estimate $C_i \lesssim \varepsilon_i^{\frac{1-\beta}{r}}$, then Hölder's inequality and \cref{thm:Lp-frames} would be enough to conclude. We are not quite able to prove such estimate, but this is morally true on average: this is the main idea that underlines the following computations. We decompose $C_i=C_i^+ + C_i^-$ with  
\[
    C_i^+ := \lambda_i^{-n\theta}\left(\sum_{j>i}\lambda_j^{\frac{n\theta}{n-1}} \left(\int_{[0,1]}A_i^{\gamma n}(t)A_j^{{(1-\gamma)n}}(t)\mathrm{d}\mu(t)\right)^{\frac{1}{n-1}}\right)^{n-1},
\]
\[
    C_i^- := \lambda_i^{-n\theta}\left(\sum_{j=0}^{i}\lambda_j^{\frac{n\theta}{n-1}} \left(\int_{[0,1]}A_i^{\gamma n}(t)A_j^{(1-\gamma)n}(t)\mathrm{d}\mu(t)\right)^{\frac{1}{n-1}}\right)^{n-1}.
\]
We start with estimating $C_i^+$. To this end we use the Hölder inequality in $t$ with exponents $\frac{1}{\tau}$ and $\frac{1}{(1-\tau)}$ for some $\tau \in (0,1)$ to be chosen later. We find
\begin{align*}
    \int_{[0,1]}A_i^{\gamma n}(t)A_j^{(1-\gamma)n}(t)\mathrm{d}\mu(t) &\lesssim \|A_i\|_{L^{\frac{\gamma n}{\tau}}(\mathrm{d}\mu)}^{\gamma n}\|A_j\|_{L^{\frac{(1-\gamma) n}{(1-\tau)}}(\mathrm{d}\mu)}^{(1-\gamma)n} \\ 
    &\lesssim \left(\varepsilon_i^{\frac{(1-\beta)\tau}{\gamma n}}\lambda_i^{\alpha\beta (\frac{1}{n}-\frac{\tau}{\gamma n})}\right)^{\gamma n} \left(\varepsilon_j^{\frac{(1-\beta)(1-\tau)}{(1-\gamma)n}}\lambda_j^{\alpha\beta (\frac{1}{n}-\frac{1-\tau}{(1-\gamma)n})}\right)^{(1-\gamma)n} \\
    &\lesssim \varepsilon_i^{(1-\beta)\tau}\lambda_i^{\alpha\beta(\gamma - \tau)} \varepsilon_j^{(1-\beta)(1-\tau)}\lambda_j^{\alpha\beta(\tau - \gamma)},
\end{align*}
where we have used \eqref{eq.Ak-bound}. Finally, we arrive at
\begin{align*}
    C_i^+ &\lesssim \lambda_i^{-n\theta + \alpha \beta (\gamma - \tau)} \varepsilon_i^{(1-\beta)\tau}\left(\sum_{j>i} \lambda_j^{\frac{n\theta +\alpha \beta (\tau - \gamma)}{n-1}}\varepsilon_j^{\frac{(1-\tau)(1-\beta)}{n-1}}\right)^{n-1} \\
    & = \lambda_i^{(n-1)\kappa} \varepsilon_i^{(1-\beta)\tau}\left(\sum_{j>i} \lambda_j^{-\kappa}\varepsilon_j^{\frac{(1-\tau)(1-\beta)}{n-1}}\right)^{n-1},  
\end{align*} 
with $\kappa := \frac{-n\theta + \alpha \beta (\gamma - \tau)}{n-1}$ which we can assume to be positive, as indeed one can take $\gamma = \tau + \eta$ for some $\frac{n\theta}{\alpha\beta}< \eta < 1-\tau$ small enough and $\theta >0$ accordingly.
Applying \cref{lem.sequence} yields 
\begin{equation}\label{ep7b1}
    \sum_{i\geq 0} \|f_{i}\|_{L^{\frac{n}{\beta}}(\mathrm{d}\nu_{\alpha -1})}^{n}C_{i}^{+} \lesssim \sum_{k_{1}, \ldots, k_{n-1}\geq 0} \sum_{i\geq 0} \|f_{i}\|^{n}_{L^{\frac{n}{\beta}}(\mathrm{d}\nu_{\alpha - 1})}\prod_{j=1}^{n-1}\left[\varepsilon_{i}^{\frac{\tau(1-\beta)}{(n-1)}}\delta_{k_{j}}(i)\left(\frac{\lambda_{i}}{\lambda_{i+k_{j}A}}\right)^{\kappa}\right],
\end{equation}
where the $\delta_{k_j}(i)$ are defined in \cref{lem.sequence}. Since the sequence $\{\lambda_{k}\}_{k\geq 0}$ is lacunary, there is a $q >1$ such that 
\begin{equation}\label{ep8b1}
    \prod_{j=1}^{n-1}\left(\frac{\lambda_{i}}{\lambda_{i+k_{j}A}}\right)^{\kappa} \lesssim \prod_{j=1}^{n-1}\frac{1}{q^{\kappa A k_{j}}}.
\end{equation}
Using \eqref{ep8b1} in \eqref{ep7b1} provides us with 
\begin{equation}\label{ep9b1}
    \sum_{i\geq 0} \|f_{i}\|_{L^{\frac{n}{\beta}}(\mathrm{d}\nu_{\alpha - 1})}^{n}C_{i}^{+} \lesssim \sum_{k_{1},\ldots, k_{n-1}\geq 0}\prod_{j=1}^{n-1}\frac{1}{q^{k_{j}A\kappa}}\sum_{i\geq 0}\|f_{i}\|^{n}_{L^{\frac{n}{\beta}}(\mathrm{d}\nu_{\alpha - 1})}\prod_{j=1}^{n-1}\left[\varepsilon_{i}^{\frac{\tau(1-\beta)}{(n-1)}}\delta_{k_{j}}(i)\right].
\end{equation}
Next, we use the Hölder inequality with exponent $\frac{1}{\beta}$ and $\frac{1}{1-\beta}$ to bound 
\begin{multline*}
    \sum_{i\geq 0} \|f_{i}\|^{n}_{L^{\frac{n}{\beta}}(\mathrm{d}\nu_{\alpha - 1})}\prod_{j=1}^{n-1}\left[\varepsilon_{i}^{\frac{\tau(1-\beta)}{(n-1)}}\delta_{k_{j}}(i)\right]\leq \left(\sum_{i\geq 0} \|f_{i}\|^{\frac{n}{\beta}}_{L^{\frac{n}{\beta}}(\mathrm{d}\nu_{\alpha - 1})}\right)^{\beta}\left(\sum_{i\geq 0}\prod_{j=1}^{n-1}\left[\varepsilon_{i}^{\frac{\tau}{(n-1)}}\delta_{k_{j}}(i)^{\frac{1}{(1-\beta)}}\right]\right)^{1-\beta} \\ 
    \lesssim  \|f\|^{n}_{L^{\frac{n}{\beta}}(\mathrm{d}\nu_{\alpha - 1})}\underbrace{\left(\sum_{i\geq 0}\prod_{j=1}^{n-1}\left[\varepsilon_{i}^{\frac{\tau}{(n-1)}}\left(\sum_{i+k_{j}A\leq \ell<i+(k_{j}+1)A} \varepsilon_{\ell}\right)^{\frac{(1-\tau)}{(n-1)}}\right]\right)^{1-\beta}}_{M}
\end{multline*}
where the last inequality stems for an application of \cref{thm:Lp-frames}. Then, since $(n-1)\left(\frac{\tau}{n-1}+\frac{1-\tau}{n-1}\right)$, we can use the Hölder inequality $\prod_{j=1}^{n-1}L^{\frac{n-1}{\tau}} \times L^{\frac{n-1}{1-\tau}}$ to bound $M$:
\[
    M \lesssim \prod_{j=1}^{n}\left(\sum_{i\geq 0} \varepsilon_{i}\right)^{\frac{\tau(1-\beta)}{(n-1)}}\left(\sum_{i\geq 0}\sum_{i+k_{j}A\leq \ell<i+(k_{j}+1)A}\varepsilon_{\ell}\right)^{\frac{(1-\tau)(1-\beta)}{(n-1)}}.
\]
Observe that 
\[
    \sum_{i\geq 0}\sum_{i+k_{j}A\leq \ell<i+(k_{j}+1)A}\varepsilon_{\ell} \leq \sum_{\ell\geq k_{j}A}\varepsilon_{\ell}\sum_{\ell-(k_{j}+1)A<i\leq \ell-k_{j}A}1 \leq A\sum_{\ell\geq 0}\varepsilon_{\ell},
\]
so that 
\[
    M\lesssim \left(\sum_{\ell\geq 0}\varepsilon_{\ell}\right)^{\tau (1-\beta)}\left(\sum_{\ell\geq 0}\varepsilon_{\ell}\right)^{(1-\tau)(1-\beta)} = \left(\sum_{\ell\geq 0}\varepsilon_{\ell}\right)^{1-\beta} = C_{\varepsilon}^{1-\beta}.
\]
It follows that 
\begin{multline}\label{c2b1} 
    \sum_{i\geq 0} \|f_{i}\|_{L^{\frac{n}{\beta}}(\mathrm{d}\nu_{\alpha - 1})}^{n}C_{i}^{+} \lesssim \|f\|^{n}_{L^{\frac{n}{\beta}}(\mathrm{d}\nu_{\alpha - 1})} \sum_{k_{1},\ldots, k_{n-1}\geq 0}\prod_{j=1}^{n-1}\frac{1}{q^{k_{j}A\kappa}}\\
    \lesssim C_{\varepsilon}^{1-\beta}\|f\|^{n}_{L^{\frac{n}{\beta}}(\mathrm{d}\nu_{\alpha - 1})}\left(\sum_{k\geq 0}\frac{1}{q^{kA\kappa}}\right)^{n-1} \lesssim C_{\varepsilon}^{1-\beta}\|f\|^{n}_{L^{\frac{n}{\beta}}(\mathrm{d}\nu_{\alpha - 1})},
\end{multline}
since $q>1$. 

In order to estimate $C_i^-$ we take advantage of the above computations, with $\tau = \gamma$ and $\kappa = \frac{n\theta}{n-1} >0$ we arrive at  
\[
    C_i^- \lesssim \lambda_i^{-n\theta} \varepsilon_i^{(1-\beta)\gamma}\left(\sum_{j=0}^{i} \lambda_j^{\frac{n\theta}{n-1}}\varepsilon_j^{\frac{(1-\gamma)(1-\beta)}{n-1}}\right)^{n-1} = \lambda_i^{-(n-1)\kappa} \varepsilon_i^{(1-\beta)\gamma}\left(\sum_{j=0}^i \lambda_j^{\kappa}\varepsilon_j^{\frac{(1-\gamma)(1-\beta)}{n-1}}\right)^{n-1}.
\]
From \cref{lem.sequence} we infer 
\begin{equation*}
    \sum_{i\geq 0} \|f_{i}\|_{L^{\frac{n}{\beta}}(\mathrm{d}\nu_{\alpha - 1})}^{n}C_{i}^{-} \lesssim  \sum_{i\geq 0}\|f_{i}\|^{n}_{L^{\frac{n}{\beta}}(\mathrm{d}\nu_{\alpha - 1})}\sum_{k_{1}, \ldots, k_{n-1}\leq \frac{i}{A}} \prod_{j=1}^{n-1}\varepsilon_{i}^{\frac{\gamma(1-\beta)}{(n-1)}}\delta_{k_{j}}(i)\left(\frac{\lambda_{i-k_jA}}{\lambda_{i}}\right)^{\kappa}.
\end{equation*}
Since the sequence $\{\lambda_{k}\}_{k\geq 0}$ is lacunary, there is a $q > 1$ such that 
\begin{equation}\label{ep8b}
    \prod_{j=1}^{n-1}\left(\frac{\lambda_{i-k_jA}}{\lambda_{i}}\right)^{\kappa} \lesssim \prod_{j=1}^{n-1}\frac{1}{q^{\kappa A k_{j}}}.
\end{equation}
Then we can continue the estimates as in the case $C_i^+$. In the end we arrive at 
\[
    \|Tf\|_{L^{n}(\mathrm{d}\mu)}^{n} \lesssim \sum_{i\geq 0} \|f_{i}\|_{L^{\frac{n}{\beta}}(\mathrm{d}\nu_{\alpha - 1})}^{n}C_{i}^{-} + \sum_{i\geq 0} \|f_{i}\|_{L^{\frac{n}{\beta}}(\mathrm{d}\nu_{\alpha - 1})}^{n}C_{i}^{+} \lesssim \left(\sum_{k\geq 0} \varepsilon_k\right)^{1-\beta}\|f\|_{L^{\frac{n}{\beta}}(\mathrm{d}\nu_{\alpha - 1})}^{n}.
\]

\subsubsection{Case of non-integer $r$} 
Let $n$ be the natural integer such that $r \in (n,n+1)$ and write $r=\vartheta n+(1-\vartheta)(n+1)$ for some $\vartheta \in (0,1)$. We start with an application of Hölder's inequality to write: 
\[
    \|Tf\|_{L^r(\mathrm{d}\mu)}^r \lesssim \int_{[0,1]} \left(\sum_{k\geq 0} |Tf_k|\right)^r \,\mathrm{d}\mu \lesssim \int_{[0,1]} \left(\sum_{k\geq 0} |Tf_k|^{\frac{r}{n}}\right)^{n\vartheta}\left(\sum_{k\geq 0} |Tf_k|^{\frac{r}{n+1}}\right)^{(n+1)(1-\vartheta)} \,\mathrm{d}\mu, 
\]
so another application of Hölder's inequality $L^{\frac{1}{\vartheta}}(\mathrm{d}\mu) \times L^{\frac{1}{1-\vartheta}}(\mathrm{d}\mu)$ gives 
\[
    \|Tf\|_{L^r(\mathrm{d}\mu)}^r \lesssim \left(\int_{[0,1]} \left(\sum_{k\geq 0} |Tf_k|^{\frac{r}{n}}\right)^{n} \mathrm{d}\mu \right)^{\vartheta} \left(\int_{[0,1]} \left(\sum_{k\geq 0} |Tf_k|^{\frac{r}{n+1}}\right)^{n+1} \mathrm{d}\mu \right)^{1-\vartheta}.
\]
It remains to prove that 
\begin{equation}
    \label{eq.non-int-main}
    \int_{[0,1]} \left(\sum_{k\geq 0} |Tf_k|^{\frac{r}{m}}\right)^{m} \mathrm{d}\mu \lesssim C_{\varepsilon}^{1-\beta}\|f\|_{L^{\frac{r}{\beta}}}^r, 
\end{equation}
where $C_{\varepsilon} = \sum_{k\geq 0}\varepsilon_k$ and for $m\in\{n,n+1\}$.  

When $m=1$ one can use \eqref{eq.m1}. Assume now $m\geq 2$ and let us estimate \eqref{eq.non-int-main}. We start by writing, for some $\theta >0$ and $\gamma \in (0,1)$ to be chosen later: 
\[
    H_m:=\int_{[0,1]} \left(\sum_{k\geq 0} |Tf_k|^{\frac{r}{m}}\right)^{m} \mathrm{d}\mu = \int_{[0,1]} \left(\sum_{k\geq 0} \|f_k\|_{L^{\frac{r}{\beta}}(\mathrm{d}\nu_{\alpha - 1})}^{\frac{r}{m}}A_k^{\gamma\frac{r}{m}}\lambda_k^{-\theta}A_k^{(1-\gamma)\frac{r}{m}}\lambda_k^{\theta}\right)^{m} \mathrm{d}\mu
\]
with $A_k(t) = \frac{|Tf_k(t)|}{\|f_k\|_{L^{\frac{r}{\beta}}(\mathrm{d}\nu_{\alpha - 1})}}$ if $\|f_{k}\|_{L^{\frac{r}{\beta}}(\mathrm{d}\nu_{\alpha-1})} \neq 0$ and $A_k(t)=0$ otherwise. 
Next, we apply Hölder's inequality $\ell^m \times \ell^{\frac{m}{m-1}}$ in the inner summation  
\begin{align*}
    H_m &\lesssim \sum_{i \geq 0} \|f_{i}\|^r_{L^{\frac{r}{\beta}}(\mathrm{d}\nu_{\alpha - 1})}\lambda_i^{-m\theta}\int_{[0,1]} A_i^{\gamma r}(t)\left(\sum_{j \geq 0}\lambda_j^{\frac{m\theta}{m-1}} A_j^{{\frac{(1-\gamma)r}{m-1}}}(t)\right)^{m-1}\mathrm{d}\mu(t) \\
    & \lesssim \sum_{i \geq 0} \|f_{i}\|^r_{L^{\frac{r}{\beta}}(\mathrm{d}\nu_{\alpha -1})}\underbrace{\lambda_i^{-m\theta} \left(\sum_{j \geq 0}\lambda_j^{\frac{m\theta}{m-1}} \left(\int_{[0,1]}A_i^{\gamma r}(t)A_j^{(1-\gamma)r}(t)\mathrm{d}\mu(t)\right)^{\frac{1}{m-1}}\right)^{m-1}}_{C_i},
\end{align*}
where the last inequality stems for an application of the Minkowski inequality. 
We split $C_i=C_i^+ + C_i^-$ where 
\[
    C_i^+ := \lambda_i^{-m\theta}\left(\sum_{j>i}\lambda_j^{\frac{m\theta}{m-1}} \left(\int_{[0,1]}A_i^{\gamma r}(t)A_j^{{(1-\gamma)r}}(t)\mathrm{d}\mu(t)\right)^{\frac{1}{m-1}}\right)^{m-1},
\]
\[
    C_i^- := \lambda_i^{-m\theta}\left(\sum_{j=0}^{i}\lambda_j^{\frac{m\theta}{m-1}} \left(\int_{[0,1]}A_i^{\gamma r}(t)A_j^{(1-\gamma)r}(t)\mathrm{d}\mu(t)\right)^{\frac{1}{m-1}}\right)^{m-1}.
\]
Again, we start with estimating $C_i^+$, as this is the most delicate term. To this end we use the Hölder inequality in $t$ with exponents $\frac{1}{\tau}$ and $\frac{1}{(1-\tau)}$ for some $\tau \in (0,1)$ to be chosen later. A combination of \cref{interpolbis} (applied with $s=\frac{\gamma}{\tau}r$ and $s=\frac{1-\gamma}{1-\tau}r$) and \cref{prop:atom} yields 
\begin{align*}
    \int_{[0,1]}A_i^{\gamma r}(t)A_j^{(1-\gamma)r}(t)\mathrm{d}\mu(t) &\lesssim \|A_i\|_{L^{\frac{\gamma r}{\tau}}(\mathrm{d}\mu)}^{\gamma r}\|A_j\|_{L^{\frac{(1-\gamma) r}{(1-\tau)}}(\mathrm{d}\mu)}^{(1-\gamma)r} \\
    &\lesssim \left(\varepsilon_i^{\frac{(1-\beta)\tau}{\gamma r}}\lambda_i^{\alpha\beta (\frac{1}{r}-\frac{\tau}{\gamma r})}\right)^{\gamma r} \left(\varepsilon_j^{\frac{(1-\beta)(1-\tau)}{(1-\gamma)r}}\lambda_j^{\alpha\beta (\frac{1}{r}-\frac{1-\tau}{(1-\gamma)r})}\right)^{(1-\gamma)r} \\
    &\lesssim \varepsilon_i^{(1-\beta)\tau}\lambda_i^{\alpha\beta(\gamma - \tau)} \varepsilon_j^{(1-\beta)(1-\tau)}\lambda_j^{\alpha\beta(\tau - \gamma)},
\end{align*}
Therefore we have 
\begin{align*}
    C_i^+ &\lesssim \lambda_i^{-m\theta + \alpha \beta (\gamma - \tau)} \varepsilon_i^{(1-\beta)\tau}\left(\sum_{j>i} \lambda_j^{\frac{m\theta +\alpha \beta (\tau - \gamma)}{m-1}}\varepsilon_j^{\frac{(1-\tau)(1-\beta)}{m-1}}\right)^{m-1} \\
    &= \lambda_i^{(m-1)\kappa} \varepsilon_i^{(1-\beta)\tau}\left(\sum_{j>i} \lambda_j^{-\kappa}\varepsilon_j^{\frac{(1-\tau)(1-\beta)}{m-1}}\right)^{m-1},
\end{align*}
with $\kappa := -m\theta + \alpha \beta (\gamma - \tau)$ which we can assume to be positive, as indeed one can take $\gamma = \tau + \eta$ for some $\frac{m\theta}{\alpha\beta}<\eta < 1-\gamma$ small enough and $\theta >0$ accordingly. 

We are in a position to apply \cref{lem.sequence} and reproduce the estimates of the integer case to obtain the counterpart of \eqref{ep9b1}, which gives 
\begin{equation}\label{ep9b2}
    \sum_{i\geq 0} \|f_{i}\|_{L^{\frac{r}{\beta}}(\mathrm{d}\nu_{\alpha - 1})}^{r}C_{i}^{+} \lesssim \sum_{k_{1},\ldots, k_{m-1}\geq 0}\prod_{j=1}^{m-1}\frac{1}{q^{k_{j}A\kappa}}\sum_{i\geq 0}\|f_{i}\|^{r}_{L^{\frac{r}{\beta}}(\mathrm{d}\nu_{\alpha - 1})}\prod_{j=1}^{m-1}\left[\varepsilon_{i}^{\frac{\tau(1-\beta)}{(m-1)}}\delta_{k_{j}}(i)\right].
\end{equation}

Similarly fo the estimation of $C_i^{-}$, we can take advantage of the previous computations with $\tau = \gamma$ so that: 
\begin{align*}
    C_i^- &\lesssim \lambda_i^{-m\theta} \varepsilon_i^{(1-\beta)\gamma}\left(\sum_{j=0}^{i} \lambda_j^{\frac{m\theta}{m-1}}\varepsilon_j^{\frac{(1-\gamma)(1-\beta)}{m-1}}\right)^{m-1} \\
    &= \lambda_i^{-(m-1)\kappa} \varepsilon_i^{(1-\beta)\gamma}\left(\sum_{j>i} \lambda_j^{\kappa}\varepsilon_j^{\frac{(1-\gamma)(1-\beta)}{m-1}}\right)^{m-1},
\end{align*}
with $\kappa = \frac{m\theta}{m-1} >0$. Applying \cref{lem.sequence} and reproducing the computations of the integer case leads to 
\begin{equation}\label{ep9b3}
    \sum_{i\geq 0} \|f_{i}\|_{L^{\frac{r}{\beta}}(\mathrm{d}\nu_{\alpha -1})}^{r}C_{i}^{-} \lesssim  \sum_{i\geq 0}\|f_{i}\|^{r}_{L^{\frac{r}{\beta}}(\mathrm{d}\nu_{\alpha -1})}\sum_{k_{1}, \ldots, k_{m-1}\leq \frac{i}{A}} \prod_{j=1}^{m-1}\varepsilon_{i}^{\frac{\gamma(1-\beta)}{(m-1)}}\delta_{k_{j}}(i)\left(\frac{\lambda_{i-k_jA}}{\lambda_{i}}\right)^{\kappa}.
\end{equation}
From \eqref{ep9b2} and \eqref{ep9b3}, one can use the same computations as in the integer case and finally obtain 
\[
    H_m \lesssim \sum_{i\geq 0} \|f_{i}\|_{L^{\frac{r}{\beta}}(\mathrm{d}\nu_{\alpha - 1})}^{r}C_{i}^{-} + \sum_{i\geq 0} \|f_{i}\|_{L^{\frac{r}{\beta}}(\mathrm{d}\nu_{\alpha - 1})}^{r}C_{i}^{+} \lesssim \left(\sum_{k\geq 0}\varepsilon_k\right)^{1-\beta}\|f\|_{L^{\frac{r}{\beta}}((1-x)^{\alpha-1}\mathrm{d}x)}^{r}.
\]   

\subsection{Consequences of \cref{thm.interpol-supercritical}}

\begin{proof}[Proof of \cref{th5}]
The proof if very similar to that of \cref{coro.embeddings-subcritical}, therefore we only outline it. Again we write $\lambda_k$ in place of $\lambda_{n_k}$. First, remark that following \eqref{eq.Bp-use}, if $A_{\frac{p}{N},\Lambda}(\alpha,\beta)$ holds for some $p>1$, then it follows that $\imath_{\mu}$ is of restricted strong type $r$ for all $r \geq p$ and therefore is of strong type $r$ for all $r>p$ by using \cref{thm.interpol-supercritical}.

If $A_{\frac{1}{N},\Lambda}(\alpha,\beta)$ holds, then we further note that $\imath_{\mu}$ is of strong type $1$, as we can write with the help of the triangle inequality, \cref{thm:muntz-bound} and \cref{prop:atom}: 
\begin{multline*}
    \|f\|_{L^{1}(\mathrm{d}\mu)} \lesssim  \int_{[0,1]}\sum_{k\geq 0} |f_{k}(t)|\,\mathrm{d}\mu(t) \lesssim \int_{[0,1]}\sum_{k\geq 0} \|f_{k}\|_{\infty}t^{\frac{\lambda_{k}}{N}}\,\mathrm{d}\mu(t)\\
    \lesssim \int_{[0,1]}\sum_{k\geq 0} \|f_{k}\|_{L^{\frac{1}{\beta}}(\mathrm{d}\nu_{\alpha-1})}\lambda_{k}^{\frac{\alpha \beta}{n}}t^{\frac{\lambda_{k}}{N}}\,\mathrm{d}\mu(t),
\end{multline*}
which after an application of Hölder's inequality (since $0<\beta <1$) provides us with  
\begin{multline*}
    \|f\|_{L^{1}(\mathrm{d}\mu)} \lesssim \sum_{k\geq 0} \|f_{k}\|_{L^{\frac{1}{\beta}}(\mathrm{d}\nu_{\alpha-1})}\lambda_{k}^{\alpha \beta}\int_{[0,1]}t^{\frac{\lambda_{k}}{N}}\,\mathrm{d}\mu(t)\\
    \lesssim \left(\sum_{k\geq 0} \|f_{k}\|_{L^{\frac{1}{\beta}}(\mathrm{d}\nu_{\alpha-1})}^\frac{1}{\beta}\right)^{\beta}\left(\sum_{k\geq 0}\lambda_{k}^{\frac{\alpha \beta}{1-\beta}}\left(\int_{[0,1]}t^{\frac{\lambda_{k}}{N}}\,\mathrm{d}\mu(t)\right)^{\frac{1}{1-\beta}}\right)^{1-\beta} \lesssim \|f\|_{L^{\frac{1}{\beta}}(\mathrm{d}\nu_{\alpha-1})},
\end{multline*} 
where we have used \cref{thm:Lp-frames} and the $A_{\frac{1}{N},\Lambda}(\alpha,\beta)$ condition.
\end{proof}

\begin{proof}[Proof of \cref{th6}]
With slight abuse of notation we write $\lambda_k$ in place of $\lambda_{n_k}$.

The implication $(i) \Longrightarrow (ii)$ follows from a duality argument based on the use of \cref{thm:Lp-frames}. More precisely, note that each $f\in M_{\Lambda}$ of the form $f=\displaystyle\sum_{k \geq 0}a_{k}\lambda_{k}^{\frac{\alpha \beta}{p}}t^{\lambda_{k}}$ and that \cref{thm:Lp-frames} writes $\|f\|^{\frac{p}{\beta}}_{L^{\frac{p}{\beta}}(\mathrm{d}\nu_{\alpha-1})} \simeq \displaystyle\sum_{k\geq 0} |a_{k}|^{\frac{p}{\beta}}\lambda_k^{-\alpha}$. 
Without loss of generality, assume $a_k \geq 0$ and use the embedding $\ell^1 \hookrightarrow \ell^p$ which holds since $p\geq 1$: this allows us to bound using (\textit{i}),    
\begin{align*}
    \int_{[0,1]} \sum_{k\geq 0}a_{k}^{p}\lambda_{k}^{\alpha \beta}t^{p\lambda_{k}}\,\mathrm{d}\mu(t) & \leq \int_{[0,1]} \left\vert \sum_{k\geq 0} a_{k}\lambda_{k}^{\frac{\alpha \beta}{p}}t^{\lambda_{k}} \right\vert^{p} \,\mathrm{d}\mu(t) \lesssim \left(\sum_{k\geq 0} \vert a_{k} \vert ^{\frac{p}{\beta}}\right)^{\beta} = \left\|\{a_k^p\}_{k\geq 0}\right\|_{\ell^{\frac{1}{\beta}}}, 
\end{align*}
where we have used (\textit{i}) and \cref{thm:Lp-frames}. Using the duality $\ell^{\frac{1}{\beta}}\times \ell^{\frac{1}{1-\beta}}$, we finally obtain
\[
    \sum_{k\geq 0} \lambda_{k}^{\frac{\alpha \beta}{1-\beta}}\left(\int_{[0,1]}t^{p\lambda_{k}}\,\mathrm{d}\mu(t)\right)^{\frac{1}{1-\beta}}<+\infty,
\]
that is, $\mu$ satifies the condition $A_{p,\Lambda}(\alpha,\beta)$ as claimed. 

For the implication $(ii) \Longrightarrow (i)$, by assumption we know that $\mu$ satisfies the assumption $A_{p,\Lambda}(\alpha,\beta)$. Let $\varepsilon >0$ and let us prove that $\mu$ also enjoys property $A_{\varepsilon,\Lambda}(\alpha,\beta)$, at which point the conclusion follows from \cref{th5}.

Let $\rho \in (0,1)$ and $\eta >\alpha\beta$. Note that for all $t \in [0,1)$ and $\rho\in [0,1)$ there is are constants $C_{\varepsilon}, C_{\varepsilon}'>0$ such that $\frac{C_{\varepsilon}}{(1-(\rho t))^{\eta}} \leq \frac{1}{(1-(\rho t)^{\varepsilon})^{\eta}} \leq \frac{C_{\varepsilon}'}{(1-(\rho t))^{\eta}}$. Since $\{\lambda_{k}\}_{k\geq 0}$ is a lacunary and subgeometric sequence, we can apply \cite[Lemma 2.10]{GaLe} (note that we use the subgeometric assumption to use it) to write 
\begin{align*}
    F_{\eta}(\rho,1) := \int_{[0,1]}\frac{\mathrm{d}\mu(t)}{(1-\rho t)^{\eta}} &\simeq \int_{[0,1]}\sum_{k\geq 0} \lambda_{k}^{\eta}(t\rho)^{\lambda_{k}} \,\mathrm{d}\mu(t) = \sum_{k \geq 0}\lambda_{k}^{\eta}\rho^{\lambda_{k}}\int_{[0,1]}t^{\lambda_{k}} \,\mathrm{d}\mu(t) \\
    & \simeq F_{\eta}(\rho,\varepsilon) = \int_{[0,1]}\sum_{k\geq 0} \lambda_{k}^{\eta}(t\rho)^{\varepsilon\lambda_{k}} \,\mathrm{d}\mu(t) \simeq \sum_{k \geq 0}\lambda_{k}^{\eta}\rho^{\varepsilon\lambda_{k}}\int_{[0,1]}t^{\varepsilon\lambda_{k}} \,\mathrm{d}\mu(t). 
\end{align*}
Fix $\kappa>0$ such that $\eta=\alpha \beta+\kappa(1-\beta)$, which is possible since $\eta > \alpha\beta$. Then we start from  
\[
    \|F_{\eta}(\rho,1)\|_{L_{\rho}^{\frac{1}{1-\beta}}(\mathrm{d}\nu_{\kappa - 1})}^{\frac{1}{1-\beta}} \simeq \|F_{\eta}(\rho,\varepsilon)\|_{L_{\rho}^{\frac{1}{1-\beta}}(\mathrm{d}\nu_{\kappa - 1})}^{\frac{1}{1-\beta}}, 
\]
and use the above to rewrite it as 
\[
    \sum_{k\geq 0}\lambda_{k}^{\frac{\eta}{1-\beta}-\kappa}\left(\int_{[0,1]}t^{\lambda_{k}}\,\mathrm{d}\mu(t)\right)^{\frac{1}{1-\beta}} \simeq \sum_{k\geq 0}\lambda_{k}^{\frac{\eta}{1-\beta}-\kappa}\left(\int_{[0,1]}t^{\varepsilon \lambda_{k}}\,\mathrm{d}\mu(t)\right)^{\frac{1}{1-\beta}}
\]
The conclusion now comes from the fact that $\frac{\eta}{1-\beta}-\kappa=\frac{\alpha\beta}{1-\beta}$.
\end{proof}

\subsection{An example of application}\label{sec.examples-supercritical}

Let $\alpha >0$, $\beta \in (0,1)$, $\gamma=\alpha\beta$, $\mathrm{d}\mu=\mathrm{d}\nu_{\gamma - 1}$. Assume that $\Lambda = \{\lambda_k\}_{k\geq 0}$ is a lacunary sequence. Consider $p<r<q$ and assume $\frac{q}{p}(1-\beta)<1$. Let $f= \sum_{k\geq 0} a_k t^{\lambda_k}$, and define the operator $T$ by 
\[
    T(t^{\lambda_k}) := \sum_{n\geq k} (\varepsilon_k\varepsilon_n)^{\frac{1-\beta}{2p}}t^{\lambda_n}. 
\]
Let us fix $\varepsilon_k=\frac{1}{k\log ^2 k}$. 
Remark first that for $s\in\{p,q\}$ there holds  
\[
    \|T(t^{\lambda_k})\|_{L^s(\mathrm{d}\nu_{\gamma - 1})} \simeq \left(\sum_{n\geq k} (\varepsilon_k\varepsilon_n)^{\frac{s(1-\beta)}{2p}} \lambda_n^{-\gamma} \right)^{\frac{1}{s}} \simeq \varepsilon_k^{\frac{(1-\beta)}{q}}\lambda_k^{-\frac{\gamma}{s}} \simeq \varepsilon_k^{\frac{(1-\beta)}{p}}\|t^{\lambda_k}\|_{L^{\frac{s}{\beta}}(\mathrm{d}\nu_{\alpha - 1})},   
\]
where we have used that $\{\varepsilon_n^{\frac{s(1-\beta)}{2q}}\lambda_n^{-\gamma}\}_{n\geq 0}$ decays exponentially fast. It follows that $T$ satisfies the assumptions of \cref{thm.interpol-supercritical}. Hence, $T$ is of strong type $r$. This information seems however difficult to obtain by arguing directly (even in our lacunary setting) as one would typically write 
\[
    Tf(t) = \sum_{n \geq 0} \sum_{0\leq k\leq n} a_k (\varepsilon_k\varepsilon_n)^{\frac{1-\beta}{2q}} t^{\lambda_n}
\]
and apply \cref{thm:Lp-frames} to bound  
\begin{align*}
    \|Tf\|_{L^r(\mathrm{d}\nu_{\gamma - 1})} &\simeq \left(\sum_{n\geq 0} \left\vert \sum_{0\leq k \leq n} a_k (\varepsilon_k\varepsilon_n)^{\frac{1-\beta}{2p}} \right\vert^{r} \lambda_n^{-\gamma}\right)^{\frac{1}{r}}\\
    &\leq \left(\sum_{n\geq 0} \varepsilon_n^{\frac{r(1-\beta)}{2p}}\lambda_n^{-\gamma} \left(\sum_{k\geq 0}|a_k|^{\frac{r}{\beta}}\lambda_k^{-\alpha} \right)^{\beta} \left(\sum_{0\leq k \leq n} \varepsilon_k^{\frac{r(1-\beta)}{2p(r-\beta)}}\lambda_k^{\frac{\alpha\beta}{r-\beta}}\right)^{r-\beta}\right)^{\frac{1}{r}} \\ 
    &\leq \left(\sum_{n\geq 0} \varepsilon_n^{\frac{r(1-\beta)}{p}}\right)^{\frac{1}{r}} \left(\sum_{k\geq 0}|a_k|^{\frac{r}{\beta}}\lambda_k^{-\alpha} \right)^{\frac{\beta}{r}}, 
\end{align*}
where we have used Hölder's inequality and the exponential motonicity of the sequence $\{\varepsilon_k^{\frac{r(1-\beta)}{2p(r-\beta)}}\lambda_k^{\frac{\alpha\beta}{r-\beta}}\}_{k\geq 0}$. Remark that this can be rewritten as 
\[
    \|Tf\|_{L^r(\mathrm{d}\nu_{\gamma - 1})} \lesssim \left(\sum_{n\geq 0} \varepsilon_n^{\frac{r(1-\beta)}{p}} \right)^{\frac{1}{r}}\|f\|_{L^{\frac{r}{\beta}}(\mathrm{d}\nu_{\alpha - 1})}, 
\]
but that the series in $n$ diverges because $\frac{r}{q}(1-\beta)<\frac{q}{p}(1-\beta)<1$.

\subsection{A counter-example}\label{sec.counter-examples-supercritical}

We consider \eqref{def.T-kernel} and $f_N(t) = \sum_{k=1}^{N} a_k t^{\lambda_k}$ with 
\[
    c_n(k)= \frac{\mathbf{1}_{1\leq n \leq k}}{n^{\frac{1+\varepsilon}{r}}k^{\frac{(1-\beta)(1+\eta)}{r}}}\lambda_k^{-\frac{\alpha\beta}{r}}\lambda_n^{\frac{\gamma}{r}} \quad\text{and } a_k=\lambda_k^{\frac{\alpha\beta}{r}}\mathbf{1}_{1\leq k \leq N},
\]
where $r>1$, $N\geq 1$ which will be taken large enough, and $\varepsilon, \eta >0$ will be chosen small enough. We also take $\mathrm{d}\mu=\mathrm{d}\nu_{\gamma - 1}$.  
One can check that by applying \cref{thm:Lp-frames} we have 
\begin{align*}
    \|T(t^{\lambda_k})\|_{L^r(\mathrm{d}\nu_{\gamma - 1})} &\simeq \left(\sum_{n\geq 1} \lambda_n^{-\gamma} |c_n(k)|^r\right)^{\frac{1}{r}}\\ 
    &\simeq \left(\sum_{n=1}^k \frac{1}{n^{1+\varepsilon}}\right)^{\frac{1}{r}} \left(\frac{1}{k^{1+\eta}}\right)^{\frac{1-\beta}{r}}\|t^{\lambda_k}\|_{L^{\frac{r}{\beta}}(\mathrm{d}\nu_{\alpha -1})} \simeq \varepsilon_k^{\frac{1-\beta}{r}}\|t^{\lambda_k}\|_{L^{\frac{r}{\beta}}(\mathrm{d}\nu_{\alpha - 1})}, 
\end{align*}
with $\sum_{k\geq 1}\varepsilon_k < \infty$. We also have 
\[
    \|f_N\|_{L^{\frac{r}{\beta}}(\mathrm{d}\nu_{\alpha -1})} \simeq \left(\sum_{k=1}^N \lambda_k^{-\alpha} |a_k|^{\frac{r}{\beta}}\right)^{\frac{\beta}{r}} \sim N^{\frac{\beta}{r}}     
\]
and 
\[
    \|Tf_N\|_{L^r(\mathrm{d}\nu_{\gamma-1})}\simeq \left(\sum_{n\geq 1}\lambda_n^{-\gamma}\left\vert\sum_{k\geq 1} a_k c_n(k)\right\vert^r\right)^{\frac{1}{r}} \simeq \left(\sum_{n=1}^N\frac{1}{n^{1+\varepsilon}}\left(\sum_{k=n}^N\frac{1}{k^{\frac{(1-\beta)(1+\eta)}{r}}}\right)^r\right)^{\frac{1}{r}}. 
\]
We lower bound: 
\[
    \|Tf_N\|_{L^r(\mathrm{d}\nu_{\gamma-1})}\gtrsim \left(\sum_{\frac{N}{3}\leq n\leq \frac{2N}{3}} \frac{1}{n^{1+\varepsilon}}\left(\sum_{k=n}^N\frac{1}{k^{\frac{(1-\beta)(1+\eta)}{r}}}\right)^r\right)^{\frac{1}{r}} \simeq N^{\beta-\frac{\varepsilon}{r}-(1-\beta)\eta}  
\]
which proves that $T$ is not of strong type $r$ as soon as we take $r>1$, and provided $\varepsilon \ll 1$, $\eta \ll 1$ and $N\gg 1$. 

\subsection{Necessity of the summability condition for a positive operator in the lacunary case}\label{sec.necessity}

We consider a fixed lacunary sequence $\Lambda = \{\lambda_k\}_{k\geq 0}$. Fix $\alpha, \beta, r$ as in \cref{thm.interpol-supercritical} and consider a sublinear positive operator $T$ of strong type $r$. Consider $\{b_k\}_{k\geq 1} \in \ell_k^{\frac{1}{\beta}}$ with positive coefficients, and define $a_k=\lambda_k^{\frac{\alpha\beta}{r}}b_k^{\frac{1}{r}}$, as well as the function  
\[
    f(t)=\sum_{k\geq 1} a_k t^{\lambda_k}. 
\]
Let $\varepsilon_k^{\frac{1-\beta}{r}} := \frac{\|T(t^{\lambda_k})\|_{L^r(\mathrm{d}\mu)}}{\|t^{\lambda_k}\|_{L^{\frac{r}{\beta}}(\mathrm{d}\nu_{\alpha - 1})}} \simeq \|T(t^{\lambda_k})\|_{L^r(\mathrm{d}\mu)} \lambda_k^{\frac{\alpha\beta}{r}}$. Therefore, we have: 
\begin{multline*}
    \sum_{k\geq 1} b_k\varepsilon_k^{1-\beta} \lesssim \sum_{k\geq 1} a_k^r\|T(t^{\lambda_k})\|_{L^r(\mathrm{d}\mu)}^r = \|\{a_kT(t^{\lambda_k})\}_{k\geq 1}\|^r_{\ell^r_kL^r(\mathrm{d}\mu)}  \\
    =\|\{a_kT(t^{\lambda_k})\}_{k\geq 1}\|^r_{L^r(\mathrm{d}\mu)\ell^r_k}   \leqslant \|\{a_kT(t^{\lambda_k})\}_{k\geq 1}\|^r_{L^r(\mathrm{d}\mu)\ell^1_k} \leqslant \|T(f)\|^r_{L^r(\mathrm{d}\mu)},
\end{multline*}
where we have used that $T$ is positive.
Since $T$ is an operator of strong type $r$ this implies 
\[
    \sum_{k\geq 1} b_k\varepsilon_k^{1-\beta} \lesssim \left(\sum_{k\geq 1} \frac{a_k^{\frac{r}{\beta}}}{\lambda_k^{\alpha}}\right)^{\beta} = \left(\sum_{k\geq 1} b_k^{\frac{1}{\beta}}\right)^{\beta},
\]
which proves by duality that $\{\varepsilon_k^{1-\beta}\}_{k\geq 1} \in \ell_k^{\frac{1}{1-\beta}}$, that is $\sum_{k\geq 1} \varepsilon _k < \infty$ as claimed. 

\bibliographystyle{alpha}
\newcommand{\etalchar}[1]{$^{#1}$}

\end{document}